\documentclass[psamsfonts, reqno]{amsart}
\usepackage{amssymb,amsfonts}
\usepackage[all,arc]{xy}
\usepackage{enumerate}
\usepackage{mathrsfs}
\usepackage{pstricks}
\usepackage{pst-node}

\usepackage{times}
\usepackage[T1]{fontenc}
\usepackage{mathrsfs}
\usepackage{epsfig}
\usepackage{color}

\newtheorem{thm}{Theorem}[section]
\newtheorem{cor}[thm]{Corollary}
\newtheorem{prop}[thm]{Proposition}

\newtheorem{conj}[thm]{Conjecture}

\newtheorem{alt}{Theorem}[section]
\theoremstyle{definition}
\newtheorem{defn}[thm]{Definition}

\newtheorem{exmp}[alt]{Example}

\theoremstyle{remark}
\newtheorem{rem}[thm]{Remark}

\newcommand{\sA}{{\mathcal A}}
\newcommand{\sC}{{\mathcal C}}

\newcommand{\sG}{{\mathcal G}}
\newcommand{\sI}{{\mathcal I}}

\newcommand{\sP}{{\mathcal P}}

\newcommand{\arr}{{\alpha}}
\newcommand{\ddd}{{d}}
\newcommand{\bbb}{{b}}
\newcommand{\sss}{{s}}

\newcommand{\NN}{{\mathbb N}}
\newcommand{\QQ}{{\mathbb Q}}
\newcommand{\ZZ}{{\mathbb Z}}
\numberwithin{equation}{section}

\newcommand{\stdnodesep}{3}
\newcommand{\doffset}{3pt}

\newcommand{\node}[3]{\rput{0}(#2){\ovalnode{#3#1}{\Large #1}}}

\newcommand{\aline}[3]{%
	\ncline[nodesepA=\stdnodesep,nodesepB=\stdnodesep]%
	{->}{#1}{#2}%
	\Aput{#3}%
}

\newcommand{\bline}[3]{%
	\ncline[nodesepA=\stdnodesep,nodesepB=\stdnodesep]%
	{->}{#1}{#2}%
	\Bput{#3}%
}

\newcommand{\dline}[4]{%
	\ncarc[nodesepA=\stdnodesep,nodesepB=\stdnodesep,offset=\doffset]%
	{->}{#1}{#2}%
	\Aput{#3}%
	\ncarc[nodesepA=\stdnodesep,nodesepB=\stdnodesep,offset=\doffset]%
	{->}{#2}{#1}%
	\Aput{#4}%
}

\newcommand{\bcircle}[3]{%
	\nccircle[angleA=#2,nodesepA=\stdnodesep]{->}{#1}{20pt}%
	\Bput{#3}%
}

\title{Intersections of multiplicative translates of \\$3$-adic Cantor sets II: Two infinite families}

\author{William C. Abram}
\author{Artem Bolshakov}
\author{Jeffrey C. Lagarias}
\thanks{The first author received support from an NSF Graduate Research Fellowship. The 
third author received support from NSF grants  DMS-1101373 and DMS-1401224.}
\address{Department of Mathematics, Hillsdale College, Hillsdale, MI 49242-1205, USA}
\email{wabram@hillsdale.edu}
\address{College of the School of Natural Sciences and Mathematics, University of Texas \newline $\text{ }$ $\text{ }$ $\text{ }$   at Dallas, Richardson, TX 75080-3021, USA}
\email{atb130030@utdallas.edu}
\address{Department of Mathematics, University of Michigan,
Ann Arbor, MI 48109-1043,USA}
\email{lagarias@umich.edu}

\date{December 5,  2015}

\begin{document}

\begin{abstract}
This paper studies the structure of finite intersections  
of general
multiplicative translates 
${\sC}(M_1, M_2, \ldots, M_n) = \frac{1}{M_1} \Sigma_{3, \bar{2}} \cap \cdots \cap \frac{1}{M_n} \Sigma_{3, \bar{2}}$ for
integers $1 \leq M_1 < M_2 < \cdots < M_n$,
in which $\Sigma_{3, \bar{2}}$ denotes
the $3$-adic Cantor set  (of $3$-adic integers whose expansions omit  the digit $2$), which
has Hausdorff dimension $\log_3 2 \approx 0.630929$.
This study was  motivated by questions concerning the discrete dynamical system on the $3$-adic integers $\ZZ_3$ given
by multiplication by $2$.
The  exceptional set  $\mathcal{E}(\mathbb{Z}_3)$ is defined to be  the set of all elements of $\mathbb{Z}_3$
 whose forward orbits under  this action  intersect the $3$-adic Cantor set $\Sigma_{3, \bar{2}}$ 
 infinitely many times.
 It is conjectured that it has Hausdorff dimension $0$. 
 An earlier paper showed that upper bounds on the  Hausdorff dimension of the exceptional set   can be extracted from knowing 
 Hausdorff dimensions of sets of the kind above, in cases where all $M_i$
 are powers of $2$. These intersection sets were shown to be fractals whose points have $3$-adic
 expansions describable by labeled paths in a finite automaton,  whose Hausdorff dimension is exactly computable and is of
 the form $\log_3(\beta)$ where $\beta$ is a real algebraic integer. 
 It gave algorithms for  determination of the automaton, and computed examples showing that the dependence of the
 automaton and the value $\beta$ on the parameters $(M_1, \ldots, M_n)$ is complicated.
 The present paper studies two new infinite families of examples, illustrating interesting behavior of the automata
 and of the Hausdorff dimension of the associated fractals. One family  has  
 associated automata whose directed graph has  a  nested sequence of
 strongly connected components of arbitrarily large depth. 
 The second family leads to  an improved upper bound for the Hausdorff dimension of the exceptional set
 $\mathcal{E}(\mathbb{Z}_3)$ of $\log_3 \phi \approx 0.438018$, where $\phi$ denotes  the Golden ratio.

\end{abstract}

\maketitle

 \tableofcontents

 

%
%
%
\section{Introduction}

Let the   $3$-adic Cantor set
 $\Sigma_3 := \Sigma_{3, \bar{2}}$ be the subset of all
$3$-adic integers whose $3$-adic expansions consist of digits $0$ and $1$ only. 
This set is a well-known fractal having Hausdorff dimension $\dim_{H}(\Sigma_{3}) = \log_3 2 \approx 0.630929$.
By a {\em multiplicative translate} of such a Cantor set we mean 
a multiplicatively rescaled set $r \Sigma_{3}= \{ r x: x\in \Sigma_3\}$,
where we restrict to  $r= \frac{p}{q} \in \QQ^{\times}$ being a rational number  that is $3$-integral, 
meaning that $r \in \ZZ_3$, or equivalently $ord_3(r) \ge 0$.
For example the  multiplicative translate $\Sigma_{3, \bar{1}}= 2 \Sigma_{3, \bar{2}}$,
which allows only  $3$-adic digits $0$ and $2$,  has the symbol structure of its digits
matching that of  ternary expansions of   the usual middle-third
Cantor set on $[0,1]$.

This paper considers sets
 given as finite intersections of such  multiplicative translates:
\begin{equation}\label{eq100}
 \sC( r_1, r_2, \cdots, r_N) := \bigcap _{i=1}^{N} \frac{1}{r_i} \Sigma_3.
\end{equation}
These sets
are fractals and this paper considers the problems of 
determining  their internal structure and of obtaining bounds on  their Hausdorff dimension.
The dependence
 of the Hausdorff dimension of the sets $\sC(r_1,\ldots, r_n)$ on the parameters $(r_1, r_2,\ldots, r_n)$ 
turns out to be  complicated and fascinating.


In  Part I  \cite{AL14c},
two of the authors presented  a method for exactly computing the Hausdorff dimension
of individual sets $\sC(r_1, \ldots, r_n)$.  
This method is suited for computer experimentation.
The method is based on the fact  
 all such sets  have a special property:
 the $3$-adic expansions of members of such a set are characterizable
by the set of all infinite paths in  a fixed labeled directed graph (finite automaton)
that emanate from a fixed initial vertex, where the edge labels are $3$-adic digits.
We term sets of this kind, characterized by a finite automaton,  {\em $3$-adic path set fractals}.
Two of the authors studied    the $p$-adic version of this concept  in \cite{AL14b},
and showed their Hausdorff dimensions are explicitly computable in terms of properties
of the associated finite automaton. 
$p$-adic path set fractals in turn are geometric realizations of objects in symbolic dynamics called {\em path sets}. Forgetting the geometric data associated to a $p$-adic path set fractal $Y$, that is, thinking of the $3$-adic digits as an alphabet with no additional structure, recovers an underlying path set $X$ which is the set of all infinite strings of digits from $\{0,1,\ldots, p-1\}$ corresponding to elements of $Y$.
 The path set underlying the $3$-adic path set fractal $\sC(r_1,\ldots, r_n)$ is denoted $X(r_1, \ldots, r_n)$,
 and will play a role in the results of this paper.  The papers \cite{AL14b}, \cite{AL14c}
 gave between them  algorithms to effectively compute $X(r_1, \ldots, r_n)$
 when given $(r_1, r_2, ..., r_n)$. 
Section \ref{sec200}  reviews basic results  on path sets and $p$-adic path set fractals;
a general theory of path sets was previously developed by two of the authors in \cite{AL14a}.

This paper is concerned with  the case $\sC(1, M)$ for $M$ a positive  integer.
The Hausdorff dimension $\dim_{H}(\sC(1, M))$ has a clear dependence on certain simple properties
of the ternary expansion  $(M)_3$ of $M$. For  
 example Part I observed:
 \begin{enumerate}
 \item[(i)]
  $\dim_{H}(C(1, M)) =0$ whenever the last ternary digit of $(M)_3$ is a $2$, i.e.
$M \equiv 2 \, (\bmod \, 3)$.
\item[(ii)]
 $\dim_{H}(C(1, 3M)) = \dim_{H}(\sC(1, M)).$  In consequence,  all trailing zeros in
the base $3$ expansion of $M$ may be cancelled off without changing the
Hausdorff dimension.
\end{enumerate} 
However the dependence on $M$ seems anything but simple when examined more closely.
It appears that arithmetic properties of $M$  influence both the structure
of the underlying automata and the Hausdorff dimension in extremely complex ways.
Part I treated in detail two  infinite families of $M$ whose ternary expansion $(M)_3$ had a
particularly 
simple form, where an exact answer for the Hausdorff dimension could be obtained.
\begin{enumerate}
 \item
 $M = L_k = (1^k)_3,$ that is $L_k = \frac{1}{2}( 3^k -1)$.
 It  obtained a  Hausdorff dimension formula for each $k \ge 1$
 and deduced that $\dim_{H}(C(1, L_k)) \to 0$ as $n \to \infty$ (\cite[Theorem 5.2]{AL14c}).
 \item
 $M=N_k = (10^{k-1}1)_3$, that is $N_k= 3^k+1$.
 It showed for each $k \ge 1$ that $\dim_{H}(\sC(1, N_k)) = \log_3 \phi \approx 0.438018,$
 where $\phi=\frac{1+ \sqrt{5}}{2}$  (\cite[Theorem 5.5]{AL14c}).
\end{enumerate} 
The automata associated to the second of these families displayed considerable
complexity. The automaton associated to $N_k$ had a number of states 
growing exponentially with $k$ and was strongly connected; it is remarkable that
its Perron eigenvalue could be computed exactly. 
Salient facts  on these families are collected in  Appendix  A   (Section \ref{secA0})
for easy reference.


This paper continues the study of the  sets ${\sC}(1, M)$
for various integers $M \ge 1$. 
We obtain results for two new infinite families of $M$ having  ternary expansions $(M)_3$ of
a regular form, $P_k=2 \cdot 3^k +1= (20^{k-1} 1)_3$ and $Q_k= 3^{2k}-3^k +1= (2^k0^{k-1}1)_3$; they are stated in Section \ref{sec:results}.  
When compared to the families treated in Part I, these  families reveal additional complexity in
the structure of the associated automata and the behavior of the
Hausdorff dimension. In particular the automata associated to  one of these families  are not
strongly connected; they are reducible and have
 arbitrarily large numbers of strongly connected components.
We  bound the  Hausdorff dimension of such $\sC(1, M)$ through estimation  of the Perron eigenvalue
 of the adjacency matrix of these automata.  To estimate the  Hausdorff dimension of one family, we make use of 
 an operation on path sets termed   {\em interleaving}, that we introduce in Section \ref{sec34}. 
 The structure of the automata was first guessed from computer experiments
and then proved. In addition to studying these two families the  paper   presents  
 further results from computer experiments to test the relation
 of Hausdorff dimension to particular patterns in the ternary
 expansion of  $M$.

The  original motivation for
studying questions of this kind  arose from 
 a problem of Erd\H{o}s \cite{Erd79}.
 This problem was generalized to a question over  the $3$-adic integers by 
 the third author  (\cite{Lag09}), who proposed a weaker version
 of the Erd\H{o}s problem, the 
{\em Exceptional set conjecture}, explained below,
which asserts that a certain set has Hausdorff dimension $0$.
The results of this paper yield new
information about the Exceptional set  conjecture without resolving it,
see Section \ref{sec12b}.

%
%
%
\subsection{Exceptional set conjecture and nesting constants }\label{sec11}

Erd\H os \cite{Erd79} conjectured that for every $n \geq 9$, the ternary expansion of $2^n$ does not omit the digit $2$. 
A weak version of this conjecture asserts that there are only finitely many $n$ such that the ternary expansion of $2^n$ does not omit the digit $2$. 
Both versions of this  conjecture are  open and appear difficult.

In \cite{Lag09} the third author proposed a $3$-adic generalization of this problem, as follows.
Let $\mathbb{Z}_3$ denote the $3$-adic integers, and let a $3$-adic integer $\alpha$ have 
$3$-adic expansion
\[
(\alpha)_3 := a_0 + a_1 \cdot 3 + a_2 \cdot 3^2 + \cdots ,       ~~\mbox{with all}~~ a_i \in \{ 0, 1, 2\}.
\]
It introduced the following notion.
\begin{defn} \label{defn-exceptional}
The  {\em $3$-adic exceptional set}  $\mathcal{E}(\mathbb{Z}_3)$ is given by
\[
\mathcal{E}(\mathbb{Z}_3) := 
\{\lambda \in \mathbb{Z}_3 : \text{for infinitely many $n \ge 0$ the expansion $(2^n \lambda)_3$ omits the digit $2$}\}.
\]
\end{defn}
This definition is less stringent  than the Erd\H{o}s problem in allowing
variation of  the new parameter $\lambda$.
 The weak version  of Erd\H os's conjecture above 
  is equivalent to the assertion that $1 \notin \mathcal{E}(\mathbb{Z}_3)$. 
  
 That paper proposed the following conjecture \cite[Conjecture 1.7]{Lag09}.

\begin{conj}\label{cj11}{\em (Exceptional Set Conjecture)}
The $3$-adic exceptional set $\mathcal{E}(\mathbb{Z}_3)$ 
has Hausdorff dimension zero, i.e. 
\begin{equation}
\dim_{H}(\mathcal{E}(\mathbb{Z}_3) )=0.
\end{equation}
\end{conj}

Clearly $0 \in \mathcal{E}(\mathbb{Z}_3)$,  and our state of ignorance is such 
that we do not know whether  $\mathcal{E}(\mathbb{Z}_3)=\{0\}$ or not.
In \cite{Lag09} the Exceptional Set Conjecture was approached by introducing the sets
\begin{equation}\label{eq120}
\mathcal{E}^{(k)}(\mathbb{Z}_3) := 
\{\lambda \in \mathbb{Z}_3 : \text{at least $k$ values of $(2^n \lambda)_3$ omit the digit 2}\},
\end{equation}
which yield the containment relation
\begin{equation}\label{1013}
 \mathcal{E}(\mathbb{Z}_3) \subseteq \bigcap_{k=1}^{\infty} \mathcal{E}^{(k)}(\mathbb{Z}_3).
\end{equation} 
That paper obtained the upper bound
$$
\dim_{H}( \mathcal{E}(\mathbb{Z}_3)) \le \dim_{H}(\mathcal{E}^{(2)}(\mathbb{Z}_3)) \le \frac{1}{2}.
$$
The sets $\mathcal{E}^{(k)}(\mathbb{Z}_3)$
form a nested family 
 \[
 \Sigma_{3, \bar{2}}=  \mathcal{E}^{(1)}(\mathbb{Z}_3 ) \supseteq \mathcal{E}^{(2)}(\mathbb{Z}_3) \supseteq 
  \mathcal{E}^{(3)}(\mathbb{Z}_3 ) \supseteq \cdots,
 \]
 and  are themselves expressed in terms of intersection sets \eqref{eq100} as  
\begin{equation}
 \mathcal{E}^{(k)}(\mathbb{Z}_3) = \bigcup_{0 \leq m_1 < \ldots < m_k} {\sC}(2^{m_1},\ldots,2^{m_k}).
\end{equation}
This connection motivated the  study made in \cite{AL14c} of the more general sets $C(M_1, ..., M_k)$.

\begin{defn}
The {\em (dyadic) nesting constant}
$\Gamma$  is given by
\begin{equation}\label{Gamma}
 \Gamma := \lim_{k \to \infty} \dim_{H}(\mathcal{E}^{(k)}(\mathbb{Z}_3)).
\end{equation} 
\end{defn}

The containment relation (\ref{1013}) implies  that the nesting constant upper bounds
to the Hausdorff dimension of the exceptional set, 
\begin{equation}\label{1015}
\dim_{H}(\mathcal{E}(\mathbb{Z}_3))\le  \Gamma.
\end{equation}
The third author raised the question in \cite{Lag09} whether $\Gamma =0$,
which if true would imply the Exceptional Set Conjecture.
This question is currently unanswered.

Part I  \cite[Section 1.2]{AL14c}  approached the problem of obtaining improved upper bounds
for $\Gamma$ by introducing  a relaxed upper bound $\Gamma_{\star}$,
called there  the {\em generalized nesting constant}, 
obtained by replacing ${\sC}(2^{m_1},\ldots,2^{m_k})$ with ${\sC}(1, M_1, ..., M_{k-1})$
in the definition above. That paper showed
$\Gamma \le \Gamma_{\star} \le \frac{1}{2}$, and also established  the lower bound
$$
\Gamma_{\star} \ge \frac{1}{2} \log_3 \phi \approx 0.21909.
$$
It follows that  one cannot 
 resolve whether $\Gamma=0$  or not using the relaxation $\Gamma_{\star}.$

%
%
%
\subsection{Statistics of  ternary digits and $n$-digit Hausdorff dimension constant }\label{sec12b}

A focus of this work was to 
shed  light on the Exceptional set conjecture, 
by gathering  evidence  whether there might exist  simple
statistics of the ternary expansion $(M)_3$ of a single integer $M$ which will 
predict that the Hausdorff dimension $\dim_{H}(\sC(1, M))$ 
 must go to $0$ as the value of the
statistic goes to infinity. 

In this paper we resolve this question for the 
 statistic  $\ddd_3(M)$ that counts the number of
nonzero digits in the ternary expansion of
the positive integer $(M)_3$.
This value  coincides  with the number of nonzero digits in the
$3$-adic expansion of $M$; note that a $3$-adic integer $\alpha$
has a finite number of non-zero digits if and only if it is a
non-negative integer $\alpha \in \NN$.

\begin{defn}\label{def131}
The {\em $n$-digit Hausdorff dimension constant} $\arr_n$ is given by 
$$
\arr_n := \sup_{ M \ge 1} \{ \dim_{H} ({\sC}(1, M)): \mbox{The expansion $(M)_3$ has at least $n$ nonzero ternary digits} \}.
$$
\end{defn}

By definition  the $\arr_n$ form a nonincreasing sequence of nonnegative numbers,
so that the limit 
$$
\Gamma_{\star\star} := \lim_{n \to \infty} \arr_n
$$
exists. 
Known results in  number theory, detailed in Section \ref{sec6},  imply that the
number of nonzero ternary digits of $2^n$ diverges as $n$ goes to infinity. Thus, we 
obtain an upper bound on the dyadic nesting constant 
\begin{equation}\label{nest-bound}
\Gamma \le \Gamma_{\star\star} =  \lim_{n \to \infty} \arr_n =  \inf_{n} \arr_n. 
\end{equation}
One of the infinite families studied in this paper has $d_3(M_k) \to \infty$ as $k \to \infty$
and using it  we show 
\begin{equation}\label{dbl-star-bound}
\Gamma_{\star\star} = \inf_{n} \arr_n  = \log_3\left( \frac{1+ \sqrt{5}}{2}\right) \approx 0.438018.
\end{equation}
In particular by \eqref{1015} we obtain an improved upper bound for the Hausdorff dimension of 
the exceptional set
\begin{equation}
\dim_{H}( \mathcal{E}(\ZZ_3)) \le \Gamma \le \Gamma_{\star\star}  \le \log_3\left( \frac{1+ \sqrt{5}}{2}\right) \approx 0.438018.
\end{equation}
In the opposite direction \eqref{dbl-star-bound} establishes  that the statistic $\ddd_3(M)$ does not have the property
that the Hausdorff dimension must go to $0$ as the statistic $\ddd_3(M) \to \infty$. 

The final section of the paper empirically studies   the Hausdorff dimension of ${\sC}(1, M)$
with respect to two  other simple statistics of 
the ternary expansion $(M)_3$:
the block number $\bbb_3(M)$ and intermittency $\sss_3(M)$; these satisfy $\bbb_3(M) \le \sss_3(M)$.
These are defined in Section \ref{sec7}.


%
%
%

\subsection{Roadmap}\label{sec16}

Section \ref{sec:results} states the main results.
Section \ref{sec200} reviews properties of  
$p$-adic path sets and their symbolic dynamics, drawing on \cite{AL14a} and \cite{AL14b}.
Intersections of multiplicative translates of $3$-adic 
Cantor sets are a special case of these constructions. 
Section \ref{sec34} introduces an  interleaving operation on path sets and 
analyzes its effect on Hausdorff dimension.
Section \ref{sec4}  studies the sets  ${\sC}( 1, P_k)$
for the infinite family $ P_k$, analyzes the structure of their associated automata, 
 and proves Theorems ~\ref{th109}-\ref{th109c},
 and additional results. Section \ref{secqk} studies the structure of ${\sC}(1,Q_k)$ for the infinite family $Q_k$, and proves Theorems ~\ref{thqkstruct}-\ref{th111a}.
Section \ref{sec6}  deals with results on the quantities $\arr_n$
and proves Theorems ~\ref{th1N1}-\ref{th106a}. 
Section \ref{sec7} presents empirical results on Hausdorff dimensions of $C(1,M)$
for $M$ having specified statistics of their ternary expansions $(M)_3$.  \\
Appendix A (Section \ref{secA0}) describes results for two  infinite families ${\sC}( 1, L_k)$ and ${\sC}( 1, N_k)$ 
treated in Part I  \cite{AL14c}.
Appendix B (Section \ref{sec9}) relates Hausdorff dimensions of ${\sC}(1, P_k)$ to those of ${\sC}(1, L_{k+1})$.\\

\noindent{\bf Acknowledgments.} 
 We thank Yusheng Luo for an important observation on the structure of the 
automata for the sets $P_k$,
 incorporated in  Definition ~\ref{de410} and Proposition ~\ref{pr410a}. 
 W. A.  thanks the University of Michigan, where much of this work was carried out. 
W. A. and A. B.  would also like to thank Ridgeview Classical Schools, which facilitated their collaboration. W.A. was partially supported by an NSF graduate fellowship.
J. L.  was  supported by NSF grants DMS-1101373 and DMS-1401224. Some work of J.L. on the paper
was done at ICERM, where he received support from the   Clay Foundation as a Clay Senior Scholar.
He thanks ICERM for support and good working conditions.


%
%
%

\section{Results} \label{sec:results}

The main results of this paper consist of  determination of presentations of the
 $3$-adic path sets $X(1, P_k)$ and $X(1, Q_k)$ associated  to  members of two infinite families
$\sC(1, P_k)$ and $\sC(1, Q_k)$ given below, with estimates of
their Hausdorff dimensions, 
along with  experimental results for $\dim_H( {\sC}(1, M))$ for certain other $M$ 
presented in Section \ref{sec7}.

%
%
%

\subsection{The infinite family $P_k = (20^{k-1}1)_3$}\label{sec21A}

We study  the path set structure of  families of integers having 
few nonzero ternary digits. 
The only  infinite families of numbers having  exactly two nonzero ternary digits
and $\dim_{H}({\sC}(1, N))>0$ are $N_k = 3^k +1 = (1 0^{k-1}1)_3$ 
and $P_k=  (20^{k-1}1)_3= 2\cdot 3^k +1$. The family $N_k$ 
was studied in Part I and here we study the family $P_k$.

We directly compute the Hausdorff dimensions of the first few sets $\mathcal{C}(1, P_k)$
using the algorithms of Part I to be the following.\\

%
%

\begin{minipage}{\linewidth}
\begin{center}
\begin{tabular}{|c |r | r| c  | r |}
\hline
\mbox{Path Set}  & \mbox{$P_k$} & \mbox{Vertices}
&
{\mbox{Perron eigenvalue}} & {Hausdorff dim}   \\
\hline
$\mathcal{C}(1,P_1)$ & 7  & 4&   $1.618033$ &  $0.438018$ \\
$\mathcal{C}(1,P_2)$ & 19 & 8 &   $1.465571$ &  $0.347934$  \\ 
$\mathcal{C}(1,P_3)$ &   55 & 16& $1.380278$  & $0.293358$  \\ 
$\mathcal{C}(1,P_4)$ &   163 & 32 &  $1.324718$  & $ 0.255960$  \\ 
 $\mathcal{C}(1,P_5)$ &  487 & 64 &   $1.370957$  &  $0.287191$ \\    
 $\mathcal{C}(1,P_6)$ &    1459 & 128 &  $1.388728$  & $0.298913$ \\
$\mathcal{C}(1,P_7)$ &     4375 & 256 & $1.392067$  & $0.301010$ \\
$\mathcal{C}(1,P_8)$ &   13123 &  512& $1.387961$   & $0.298408$ \\ \hline 
\end{tabular} \par
\bigskip
\hskip 0.5in {\rm TABLE 2.1.} Hausdorff dimension of $\mathcal{C}(1,P_k)$ (to six decimal places)
\newline
\newline
\end{center}
\end{minipage}

The first thing to observe
from this data  is the non-monotonic behavior of the Hausdorff dimension
as a function of $k$;
the second observation is the possibility that the dimensions are bounded away from zero.
Our results below explain both these features. 
We also  observe that $\dim_{H}(\mathcal{C}(1,P_k)) = \dim_{H}(\mathcal{C}(1,L_{k+1}))$
for $1 \le k \le 4$ but equality does not hold for $k=5$.
In an Appendix B (Section \ref{sec9}) we
show that $\dim_{H}(\mathcal{C}(1,P_k)) \ge \dim_{H}(\mathcal{C}(1,L_{k+1}))$
holds in general.

Our first result determines   properties of a presentation of the path set  $X(1, P_k)$.
The resulting directed graphs  are shown to be reducible,  having  a complicated structure
with  nested strongly connected components.  
 

\begin{thm}\label{th109} {\rm (Path set presentation for family $P_k$)}

 (1) For $P_k=2 \cdot 3^k+1= (20^{k-1}1)_3$,
the path set $X(1, P_k)$ underlying $\mathcal{C}(1,P_k)$
has a path set presentation $(\sG_k, v_0)$ 
 that has exactly $2^{k+1}$ vertices. 
 
 (2) The graph $\sG_k$ is a nested 
 sequence of $1+ \lfloor k/2\rfloor$ distinct strongly connected components. 
 
 (3) The underlying graph $G=G_k$ for $\sG_k$ has an automorphism of order $2$ and is a connected double
 cover of its quotient graph $H_k$.
\end{thm}  

The structure of $G_k$ is that
of    a ``Matryoshka doll" with a single set of nested components
at each level.  The non-monotonicity of the Hausdorff dimension as
a function of $k$   can be  related to the existence of multiple
strongly connected components in the graphs $G_k$. The non-monotonicity occurs 
  because of a switch in  which strongly connected component
has the largest topological entropy. We discuss this issue further  in Section \ref{sec53}, see Remark \ref{rem47}.


Regarding the behavior of the Hausdorff dimension as $k \to \infty$, we establish
the following result. 


\begin{thm}\label{th109c} {\rm (Hausdorff dimension bounds for family $P_k=2 \cdot 3^k+1$)}

(1) The Hausdorff dimension of $\mathcal{C}(1, P_k)$ satisfies the asymptotic lower bound
\[
\liminf_{k \to \infty}  \dim_H(\mathcal{C}(1,P_k)) \geq \frac{1}{8} \log_3 (2).
\] 

(2) Furthermore, for all $k \ge 1$, 
\[\dim_H (\mathcal{C}(1,P_k)) \geq \frac{1}{13} \log_3(2).
\]
\end{thm}  

The lower bounds in Theorem \ref{th109c} 
are obtained by further inspection of the graph associated to $\mathcal{C}(1,P_k)$. 
We also have an upper bound
\[
\dim_H (\mathcal{C}(1,P_k))  \leq \log_3 \phi.
\]
which follows from Theorem \ref{r2bound} below.

In Section \ref{sec54} we obtain 
additional results on intersection of sets in the infinite  family $P_k$  above.
We show that the Hausdorff dimensions of
arbitrarily large intersections are always positive. 
However this is no longer true if we allow intersections of sets from the infinite family $P_k$
with those of the infinite family $N_k = (10^{k-1}1)_3$ treated in \cite[Sect. 4]{AL14c} and reviewed in Appendix A (Section ~\ref{secA0}),
which also consists of numbers having exactly two nonzero ternary digits. 
For example, it is easy to show that for each $k \ge 1$,
\[ 
\mathcal{C}(1, N_k, P_k) = \{0\},
\]
 so that
 $\dim_{H}(\mathcal{C}(1, N_k, P_k)) =0$.

%
%
%

\subsection{The infinite family $Q_k = (2^k0^{k-1}1)_3$}\label{sec22A}

We next study an infinite family of integers whose 
number of nonzero ternary digits grows without bound: $Q_k = (2^k0^{k-1}1)_3 = 3^{2k}-3^k+1$.
The example $Q_2$ having a large Hausdorff dimension
was found by computer search, and led to study of this family.


\begin{thm} \label{thqkstruct} (Path set presentation for  family $Q_k$) 

(1) For $Q_k = 3^{2k}-3^k+1 = (2^k0^{k-1}1)_3$, the path set $X(1,Q_k)$ 
underlying $\mathcal{C}(1,Q_k)$ has a path set presentation $(\mathcal{G}_k,v_0)$
 that has exactly $4^k$ vertices and $6 \cdot 4^{k-1}$ edges. 
 
 (2) The underlying graph $\mathcal{G}_k$ is strongly connected.
\end{thm} 

Though the number of nonzero ternary digits of $Q_k$ grows without bound, 
 the Hausdorff dimension of $\mathcal{C}(1,Q_k)$ is constant 
independent of  $k$.


\begin{thm}\label{th111a} 
{\rm (Hausdorff dimensions for family $Q_k=3^{2k}-3^k+1$)}
For all $k \ge 2$ the Hausdorff dimension of 
$\mathcal{C}(1, Q_k)$ satisfies 
 \[
\dim_H(\mathcal{C}(1,Q_k)) = \log_3 \phi \approx 0.438018,
\]
where $\phi= \frac{1+ \sqrt{5}}{2}$.
\end{thm}  

This result is established by showing that the path set $X(1, Q_k)$ is
given by an interleaving construction from the path set $X(1, Q_1)$,
that is
$X(1, Q_k) = X(1, 7)^{(\ast k)},$
as defined in Section \ref{sec34}.
%
%
%
%

\subsection{The $n$-digit Hausdorff dimension constants  $\arr_n$.}\label{sec23A}

It is a  known fact  that the number of nonzero ternary digits in $(2^n)_3$ goes to infinity
as $n \to \infty$, i.e. for each $k \ge 2$ there are only finitely many $n$ with $(2^n)_3$ having
at most $k$ nonzero ternary digits.  Using this fact we easily deduce the following
consequence.

\begin{thm} \label{th1N1}
The nesting constant $\Gamma$ satisfies
\begin{equation}
\Gamma \le   \lim_{n \to \infty}  \arr_n.
\end{equation}
In particular
\[
\dim_{H}(\mathcal{E}(\ZZ_3)) \le  \Gamma_{**} = \lim_{n \to \infty}  \arr_n.
\]
\end{thm}

It follows that  individual values $\arr_n$ give upper bounds on $\Gamma$.

\begin{thm}\label{th106a}
We have for all $k \ge 2$ that
$$ 
\arr_{k}= \log_3 \phi \approx 0.438018,
$$
 where $\phi= \frac{1+ \sqrt{5}}{2}$
is the golden ratio. This value is attained by $\mathcal{C}(1,Q_k)$ for
$$
Q_{k} : = (2^{k} 0^{k-1} 1)_3.
$$
\end{thm}

In particular this result yields an improved upper bound on the nesting constant 
$$
\Gamma \le \log_3 \phi,
$$
and on the Hausdorff dimension of the Exceptional set. 
It also gives
$$
\Gamma_{\star\star} = \log_3 \phi \approx 0.438018.
$$
We prove Theorem \ref{th106a} in Section \ref{sec42n}.

Using the known bound 
for the generalized dyadic nesting constant $\Gamma_{\star} \le \arr_2$
 established in Part I \cite[(1.16)]{AL14c}  we obtain the following corollary.

\begin{cor}\label{th106b}
We have
$$
\Gamma_{\ast}  \le  \log_3 \phi \approx 0.438018, 
$$
in which  $\phi= \frac{1+ \sqrt{5}}{2}$
is the golden ratio.
\end{cor}

%
%
%
\subsection{Notation} \label{sec151}
The notation $(m)_3$ means either the base $3$ expansion of the positive integer $m$, or else
the $3$-adic expansion of $(m)_3$. In the $3$-adic case this expansion is to be read right to left,
so that it is compatible with the ternary expansion. That is, $\alpha = \sum_{j=0}^{\infty} a_j 3^j$
will be written $( \cdots a_2 a_1 a_0)_3$.

%
%
%

\section{Symbolic dynamics, path sets and $p$-adic path set fractals} \label{sec200}

%
%
%

\subsection{Symbolic dynamics, graphs and finite automata } \label{sec21}

The constructions of this paper are based on 
the fact that the points in  intersections of multiplicative translates of $3$-adic Cantor sets
have $3$-adic expansions that are describable in terms of allowable paths generated by
finite directed labeled graphs. 
We  use  symbolic dynamics on certain closed subsets of the one-sided shift space $\Sigma=\sA^{\NN}$ with fixed 
symbol alphabet
$\sA$, which for our application will be specialized to $\sA=\{0,1,2\}$.
A basic reference for directed graphs and symbolic dynamics, 
which we follow, is  Lind and Marcus \cite{LM95}.

By a {\em graph} we mean a finite directed graph, allowing loops and multiple edges. A {\em labeled graph}
is a graph assigning labels  to each directed edge; these labels are drawn from a finite symbol alphabet.
A labeled directed graph can be interpreted as a {\em  finite automaton}
in the sense of automata theory.
In our applications to $3$-adic digit sets, the labels are drawn from the alphabet $\sA= \{ 0, 1, 2\}.$ In a directed
graph, a vertex is a {\em source} if all directed edges touching that vertex are outgoing; it is a {\em sink} if all
directed edges touching that edge are incoming. A vertex  is {\em essential} if it is neither a source nor a sink;
and is called {\em stranded} otherwise. A graph is \emph{essential} if all of its vertices are essential. 
 A  graph $G$ is
  {\em strongly connected} if for each two vertices $i, j$ there is a directed path from $i$ to $j$.
We let $SC(G)$ denote the set of strongly connected component subgraphs of $G$.

We use some basic facts from the Perron-Frobenius theory of nonnegative matrices.
The {\em Perron eigenvalue} (\cite[Definition 4.4.2]{LM95}) of a  nonnegative  real matrix $\mathbf{A} \neq 0$ is
the largest real  eigenvalue $\beta \ge 0$ of $\mathbf{A}$.  
A nonnegative matrix is {\em irreducible} if for each row and column $(i, j)$ some power ${\bf A}^m$
has $(i,j)$-th entry nonzero.  A nonnegative matrix ${\bf A}$ is {\em primitive}  if some power ${\bf A}^k$
for an integer $k \ge 1$ has all entries positive; primitivity implies irreducibility but not vice versa.
The {\em Perron-Frobenius Theorem} \cite[Theorem 4.2.3]{LM95} for
an irreducible nonnegative matrix ${\bf A}$ states that:
\begin{enumerate}
\item
The Perron eigenvalue $\beta$ is geometrically
and algebraically simple, and has
 an everywhere positive eigenvector ${\bf v}.$
 \item
 All other eigenvalues $\mu$ have $|\mu| \le \beta$, so that $\beta = \sigma({\bf A})$,
 the spectral radius of ${\bf A}$.
 \item 
 Any other everywhere positive eigenvector must 
 be  a positive multiple of ${\bf v}$.
 \end{enumerate}
 For a general nonnegative 
real matrix $\mathbf{A} \neq 0$, the Perron eigenvalue need not be simple, but it
still equals the spectral radius $\sigma(\bf{A})$ and it has at least one  everywhere nonnegative eigenvector.

We apply this theory to adjacency matrices of graphs. 
A (vertex-vertex) {\em adjacency matrix} ${\bf A} ={\bf A}_{G}$ of the directed graph  $G$ has
entry $a_{ij}$ counting the number of directed edges from vertex $i$ to vertex $j$.
The adjacency matrix is irreducible if and only if the associated graph is strongly connected,
and we also call the graph {\em irreducible} in this case.
  Here primitivity of the adjacency matrix of 
a directed  graph $G$ is equivalent to the graph being strongly connected
and aperiodic, i. e.  the greatest common divisor of its (directed) cycle lengths is  $1$. 
For an adjacency matrix of a graph containing 
at least one directed cycle,  its Perron eigenvalue is necessarily a real algebraic integer $\beta \ge 1$
(see Lind \cite{Lin84} for a characterization of these numbers).

%
%
%

\subsection{$p$-Adic path sets, sofic shifts  and $p$-adic path set fractals} \label{sec22}

Our basic objects are special cases of the following definition. 
A {\em pointed graph} is a pair $(\sG, v)$ consisting of a directed labeled graph $\sG =(G, \mathcal{E})$
and a  marked vertex $v$ of $\sG$. Here $G$ is a (directed) graph and $\mathcal{E}$ is 
an assignment of labels $(e, \ell)= (v_1, v_2, \ell)$ to the edges of $G$, where every edge gets a single label,
and no two triples are the same (but multiple edges and loops are permitted otherwise).

\begin{defn} \label{de211}
 Given a pointed graph $(\sG, v)$ its associated
 \emph{path set} 
   $\sP = X_\mathcal{G}(v) \subset \sA^{\NN}$ is
the set of  all infinite one-sided symbol sequences $(x_0, x_1, x_2, ...) \in \sA^{\NN}$, 
giving the successive labels of all one-sided infinite walks in $\mathcal{G}$ 
issuing from the distinguished vertex $v$. 
Many different $(\mathcal{G}, v)$ may give the same path set $\sP$, and we call any such
$(\mathcal{G},v)$  a \emph{presentation} of $\sP$. 
\end{defn}

An important class of presentations have the following extra property.
We  say that a directed labeled graph $\sG =(G, v)$ is {\em right-resolving}
if for each vertex of $\sG$ all directed edges outward have
distinct labels. (In automata theory $\sG$ is called a {\em deterministic
automaton}.) One can show that every path set has a right-resolving presentation.

Note that the labeled  graph $\sG$ without a marked vertex determines
a {\em one-sided sofic shift} in the sense of symbolic dynamics, as defined in \cite{AL14a}. 
This sofic shift  comprises
the set union of the path sets at all vertices of $\sG$.
Path sets are closed sets in the shift topology,
but are in general non-invariant under the one-sided shift operator.
Those path sets $\sP$  that are invariant are exactly the one-sided sofic shifts \cite[Theorem 1.4]{AL14a}. 

We study the path set concept in symbolic dynamics in \cite{AL14a}.
 The collection of path sets $\mathcal{P}= X_{\sG}(v)$ in a given
alphabet is closed under finite union and intersection (\cite[Theorem 1.2]{AL14a}). 
The symbolic dynamics analogue of Hausdorff dimension is topological entropy. 
The {\em topological entropy} of a path set $H_{top} (\mathcal{P})$ is given by
$$
H_{top}(\mathcal{P}) := \limsup_{n \to \infty} \frac{1}{n} \log N_n(\mathcal{P}),
$$
where $N_n(\mathcal{P})$ counts the number of distinct blocks of symbols of lengh $n$
appearing in elements of $\mathcal{P}$. 
The topological entropy is easy to compute given a right-resolving presentation.
By \cite[Theorem 1.13]{AL14a}, it is 
 \begin{equation} \label{top-entropy}
 H_{top}(\mathcal{P}) = \log \beta
 \end{equation}
 where 
 $\beta$ is the Perron eigenvalue of the adjacency matrix ${\bf A}={\bf A}_G$ of 
 the underlying directed graph $G$ of $\sG$, e.g. the spectral radius of ${\bf A}$.

%
%
%

\subsection{$p$-Adic symbolic dynamics and graph directed constructions}\label{sec23}

We now suppose $\sA = \{0, 1,2, ..., p-1\}$.
We can view the elements  of 
a path set $\mathcal{P}$ on this alphabet  geometrically as describing the digits in the
$3$-adic expansion of a $3$-adic integer.  This is done using a map
$\phi: \sA^{\NN} \to \ZZ_p$
from symbol sequences into $\ZZ_p$.
We call the resulting image set $K = \phi(\mathcal{P})$ a \emph{$p$-adic path set fractal}.
Such sets are studied in \cite{AL14b}, where they are related to
graph-directed fractal constructions.
The class of  $p$-adic path set fractals  is closed under the Minkowski sum and $p$-adic addition
and multiplication by rational numbers $r \in \QQ$ that lie in $\ZZ_p$ (\cite[Theorems 1.2-1.4]{AL14b}).

It is possible to compute the Hausdorff dimension of a 
$p$-adic path set fractal  directly from a suitable presentation 
of the underlying path set $\mathcal{P}=X_{\sG}(v)$.
We will use the following result.
\begin{prop}\label{pr22a}
Let $p$ be a prime, and $K$ a set of $p$-adic integers whose allowable
$p$-adic expansions are described by the symbolic dynamics
of a $p$-adic path set  $X_K$ on symbols $\mathcal{A} =\{ 0, 1, 2, \cdots, p-1\}$.
Let $(\mathcal{G},v)$ be a presentation of
this path set that is right-resolving.

(1)  The map $\phi_p: \mathbb{Z}_p \rightarrow [0,1]$ taking 
$\alpha= \sum_{k=0}^{\infty}{a_k p^k} \in \mathbb{Z}_p$
to the real number with base $p$ expansion
 $\phi_p(\alpha) :=\sum_{k=0}^\infty \frac{a_k}{p^{k+1}}$
is a continuous map, and the image of $K$ under this map,  $K':= \phi_p(K) \subset [0,1]$, is a
graph-directed fractal in the sense of Mauldin-Williams.

 (2) The Hausdorff dimension of the $p$-adic path set fractal $K$ is 
\begin{equation}
\dim_{H}(K) = \dim_{H}(K') = \log_p  \beta,
\end{equation}
where $\beta$ is the spectral radius of the adjacency matrix ${\bf A}$ of $G$.
\end{prop}

\begin{proof}
These results are  proved in \cite[Section 2]{AL14b}.
\end{proof}

In this paper we treat the case $p=3$ with $\sA = \{ 0, 1, 2\}$.
The $3$-adic Cantor set is a $3$-adic path set fractal, so these
general properties above guarantee that the intersection of 
a finite number of multiplicative translates
of $3$-adic Cantor sets will itself be a $3$-adic path set fractal $K$,
generated from an underlying path set. 

To do calculations with such sets we will need algorithms
for converting presentations of a given $p$-adic path set to presentations
of new $p$-adic path sets derived by the operations above. 
We refer the reader to \cite{AL14b} for the $p$-adic arithmetic operations, and to \cite{AL14a} for union and intersection. A further useful operation 
called {\em interleaving} will be developed in the next subsection; 
this operation is sometimes useful in computing Hausdorff dimension.

%
%
\subsection{Interleaving operation on path sets} \label{sec34}

Let $\mathcal{P}  = X_{\mathcal{G}}(v) \subset \mathcal{A}^{\mathbb{N}}$ be a path set, and let $n$ be a positive integer. 
In the paper \cite{AL14a} the first and third authors studied a {\em decimation} operation on path sets.
Given $j \ge 0$ and $m \ge 1$, define
the {\em decimation map} $\psi_{j, m}: \sA^{\NN} \to \sA^{\NN}$ by
$$
\psi_{j, m}(a_0 a_1 a_ 2 \cdots) := (a_j a_{j+m} a_{j+2m} \cdots).
$$
The decimation  operation extracts the digits of the path set in a specified 
infinite arithmetic progression of indices. We set  
$$
\psi_{j,m}(\sP) := \{ \psi_{j, m}(x): x \in \sP\}.
$$
Here \cite[Theorem 1.5]{AL14a} proved that if $\sP$ is
a path set, then for each fixed $(j, m)$ with $j \ge 0, m\ge1$ the sets $\psi_{j,m}(\sP)$ are 
path sets. 

Here we consider a kind of inverse operator to decimation, which we term  {\em interleaving}.

\begin{defn}\label{de341}
Let $n \ge 1$ be given.
The {\em $n$-interleaving } of a closed set $\mathcal{X} \subset \sA^{\NN}$ (not necessarily a path set)
is 
\[
\mathcal{X}^{(*n)} := \{ (x_i)_{i=0}^\infty \in \mathcal{A}^{\mathbb{N}}\, : \, (x_j, x_{j+n}, x_{j+2n}, \cdots) \in \mathcal{X} \text{ for all } 0 \leq j \leq n-1\}.
\]
\end{defn}

We will show that the interleaving 
$\mathcal{P}^{(*n)}$ is itself a path set, and that its topological entropy is the same as that of $\mathcal{P}$.


\begin{prop} \label{aprop1}  
(1) For any $n \ge1$ and any path set $\mathcal{P}$, the $n$-interleaving set $\mathcal{P}^{(*n)}$ is a path set.

 (2) There is an algorithm taking $n$ and a path set presentation $\mathcal{G}$ of $\mathcal{P}$ 
 and giving a path set presentation $\mathcal{H}$ of $\mathcal{P}^{(*n)}$. If $\mathcal{G}$ has $k$ verticies and $m$ edges, then $\mathcal{H}$ has $k^n$ verticies and $m k^{n-1}$ edges.
\end{prop}

\begin{proof} 
It suffices to prove (2).
Suppose $\mathcal{P} = X_{\mathcal{G}}(v_0)$, and that the vertices of $\mathcal{G}$ are $v_0, v_1, \ldots, v_{k-1}$, so that $\mathcal{G}$ 
has $k$ vertices. Let $l_j$ be the label of vertex $v_j$ for each $0\leq j \leq k-1$. If the $l_j$ do not all have the same number of digits, 
append $0's$ to the left of labels as necessary to ensure that the labels $l_0, \ldots , l_j$ are distinct and have the same number of digits. 

The vertex set of $\mathcal{H}$ will be $V = \{ v_{i_1, i_2,\ldots, i_n} | 0 \leq i_j \leq k-1 \text{ for all j}\}$, so that $\mathcal{H}$ will have $k^n$ vertices. The vertex $v_{i_1,i_2,\ldots, i_n}$ will have label $l = l_{i_1}\star  l_{i_2} \star \cdots\star l_{i_n}$, that is, the concatenation of the labels 
of $v_{i_1}, v_{i_2}, \ldots, v_{i_n}$. Since the labels $l_j$ are all distinct and have the same number of digits, the vertex labels in $\mathcal{H}$ 
as defined will also be distinct.

Now for each edge labeled $a$ from $v_i$ to $v_j$ in $\mathcal{G}$, construct an edge labeled $a$ from $v_{i_1,i_2,\ldots, i_{n-1},i}$ to $v_{j,i_1,i_2,\ldots, i_{n-1}}$ for all $0 \leq i_1, i_2, \ldots, i_{n-1} \leq k-1$. Thus, for each edge of $\mathcal{G}$, $\mathcal{H}$ will have $k^{n-1}$ corresponding edges, 
so that if $\mathcal{G}$ has $m$ edges, then $\mathcal{H}$ has $mk^{n-1}$ edges. $\mathcal{H}$ is evidently right-resolving or strongly connected if $\mathcal{G}$ is right-resolving or strongly connected, respectively. For simplicity, we will assume from here that $\mathcal{G}$ is right-resolving.
We can do this since if $\mathcal{G}$ is not right-resolving, we can perform the right-resolving construction of \cite[Section 3]{AL14a} to obtain a right-resolving presentation of $\mathcal{P}$, and proceed with this presentation in place of $\mathcal{G}$.

We claim that $\mathcal{P}^{(*n)} = X_{\mathcal{H}}(v_{0,0,\ldots,0})$. First we will show that $\mathcal{P}^n \subseteq X_{\mathcal{H}}(v_{0,0,\ldots,0})$. Suppose $(x_t)_{t=0}^\infty \in \mathcal{P}^n$. Then there must be elements 
\[(x_{0,t})_{t=0}^\infty, (x_{1,t})_{t=0}^\infty, \ldots, (x_{n-1,t})_{t=0}^\infty \in \mathcal{P}
\]
 such that $x_{j,t} = x_{nt+j}$ for all $0 \leq j \leq n-1$ and $0 \leq t < \infty$. Since $\mathcal{G}$ is right-resolving, each of these elements of $\mathcal{P}$ corresponds to a unique infinite vertex path $v_0,v_{i_{j,0}},v_{i_{j,1}}, \ldots$ in $\mathcal{G}$.  We can traverse an initial path in the pointed graph $\mathcal{H}(v_{0,0,0,\ldots, 0})$ with labels \newline $x_0, x_1,\ldots, x_{n-1}$, since there are edges with each of these labels emanating from $v_0$ in $\mathcal{G}$. 
 This path takes us to the vertex $v_{i_{n-1,0},i_{n-2,0},\ldots, i_{0,0}}$. Since there is a vertex labeled $x_{n+j}$ emenating fom vertex $v_{i_{j,0}}$
  and going to $v_{i_{j,1}}$ for all $0 \leq j \leq n-1$, we can extend our path to a path labeled $x_0, x_1,\ldots, x_{2n-1}$ beginning at $v_{0,0,\ldots,0}$ and ending at $v_{i_{n-1,1},i_{n-2,1},\ldots, i_{0,1}}$.

Inductively, assume we have constructed a path with labels $x_0, x_1, \ldots, x_{rn-1}$ in $\mathcal{H}$ originating at $v_{0,0, \ldots, 0}$ and terminating at $v_{i_{n-1,r-1},i_{n-2,r-1},\ldots, i_{0,r-1}}$. Then since there is an edge in $\mathcal{G}$ labeled $x_{rn+j}$ from $v_{j,r-1}$ to $v_{j,r}$, 
we can extend our path to a path labeled $x_0, x_1, \ldots, x_{(r+1)n-1}$ terminating at $v_{i_{n-1,r},i_{n-2,r},\ldots, i_{0,r}}$. 
Thus, there is an infinite path in $\mathcal{H}$ originating at $v_{0,0,\ldots, 0}$ with label $(x_0, x_1, x_2, \ldots)$, so $(x_i)_{i=0}^\infty \in X_{\mathcal{H}}(v_{0,0,\ldots,0})$, hence $\mathcal{P}^n \subseteq  X_{\mathcal{H}}(v_{0,0,\ldots,0})$. 

Now to show  $X_{\mathcal{H}}(v_{0,0,\ldots,0}) \subseteq \mathcal{P}^n$: Suppose $(x_i)_{i=0}^\infty$ is an element of 
$X_{\mathcal{H}}(v_{0,0,\ldots,0})$. Then there is a vertex path $v_{0,0,\ldots, 0} ; v_{i_0,0,\ldots,0}; v_{i_1,i_0,0,\ldots,0}; \ldots; v_{i_{n-1}, i_{n-2};\ldots, i_0}; \ldots$ in $\mathcal{H}$ which can be traversed by edges labeled $x_0,x_1,\ldots$. 
Notice that the first coordinate of a vertex must be the last coordinate of the vertex that follows after $n-1$ steps. Since the initial vertex is $v_{0,0,\ldots,0}$, we know that for each $0 \leq j \leq n-1$, there is an edge in $\mathcal{G}$ labeled $x_j$ from $v_0$ to $v_{i_j}$. For any $j < \infty$, an edge in $\mathcal{H}$ labeled $x_j$ from $v_{i_1,i_2,\ldots, i_n}$ to $v_{i_{n+1},i_1,i_2,\ldots, i_{n-1}}$ corresponds to an edge in $\mathcal{G}$ labeled $x_j$ fom $v_{i_n}$ to $v_{i_{n+1}}$. Following our path in $\mathcal{H}$ for $n-1$ more steps gets us to a vertex whose last coordinate is $i_{n+1}$, so the edge in $\mathcal{H}$ labeled $x_{n+j}$ emanating from this vertex corresponds to an edge in $\mathcal{G}$ labeled $x_{n+j}$ emanating from $v_{i_{n+1}}$. Thus, for each $0 \leq j \leq n-1$, the labels $(x_j, x_{j+n} x_{j+2n}, \ldots)$ are the labels of an infinite path in $\mathcal{G}$ originating at $v_0$, so $(x_i)_{i=0}^\infty \in \mathcal{P}^n$, hence $X_{\mathcal{H}}(v_{0,0,\ldots,0}) \subseteq \mathcal{P}^n$, as desired.
\end{proof}


\begin{rem} \label{rem35}
(1) The presentation $\mathcal{H}$ of $\mathcal{P}^{(*n)}$ given in
the proof above  is right-resolving  (resp. strongly connected) if and only if the presentation $\mathcal{G}$ of $\mathcal{P}$ used in its construction is right-resolving (resp. strongly connected).

(2) The operation of interleaving can be extended to interleave several different sets 
$$
\sI(X_1, X_2, ..., X_m) := \{ x \in \sA^{\NN}: \psi_{j, m}(x) \in X_i \quad \mbox{for} \quad 0 \le j \le m-1.\}
$$
One can show that if each $X_i = \sP_i$ is a path set then $\sI(\sP_1, \sP_2, \cdots, \sP_n)$
is a path set. 
\end{rem}

We next show that the  $n$-interleaving operation $\mathcal{P}^{(*n)}$ has the nice feature
that it  preserves topological entropy. Following \cite{AL14a}
we  define the \emph{path topological entropy} $H_p(\mathcal{P})$ of a path set $\mathcal{P}$ by
\begin{equation} \label{apeq2} 
H_p(\mathcal{P}) := \limsup_{k \rightarrow \infty} \frac{1}{k} \log N_k^I(\mathcal{P}),
\end{equation}
where $N_k^I(\mathcal{P})$ is the number of \emph{initial} blocks of length $k$ from $\mathcal{P}$, then \cite[Theorem 1.11]{AL14a} shows that
\begin{equation}\label{apeq3} H_p(\mathcal{P}) = H_{top}(\mathcal{P}),
\end{equation}
and that the $\limsup$'s are obtained as limits.


\begin{prop} \label{aprop3} 
If $\mathcal{P}$ is a path set, then
\begin{equation} \label{apeq1} H_{top}(\mathcal{P}^{(*n)} )= H_{top}(\mathcal{P}).
\end{equation}
\end{prop}

\begin{proof} Using (\ref{apeq3}), it suffices to show that $\mathcal{P}$ and $\mathcal{P}^{(*n)}$ have the same path entropy. But we can see directly from the definition of $\mathcal{P}^{(*n)}$ that $N_{nk}^I(\mathcal{P}^{(*n)}) = (N_k^I(\mathcal{P}))^{n}$, since an initial path of length $nk$ in $\mathcal{P}^{(*n)}$ corresponds to $n$ (not necessarily distinct) initial paths of length $k$ in $\mathcal{P}$. Thus,
\begin{align*}
H_p(\mathcal{P}^{(*n)}) &= \lim_{k \rightarrow \infty} \frac{1}{k} \log N_k^I(\mathcal{P}^{(*n)})\\
&= \lim_{k \rightarrow \infty} \frac{1}{nk} \log N_{nk}^I(\mathcal{P}^{(*n)}) \\
&= \lim_{k \rightarrow \infty} \frac{1}{nk} \log [(N_k^I(\mathcal{P}))^{n}] \\
&= \lim_{k \rightarrow \infty} \frac{1}{k} \log N_k^I(\mathcal{P}) = H_p(\mathcal{P}),
\end{align*}
as desired.
\end{proof}

If $\mathcal{A}=\{0,1,\ldots,p-1\}$, let $\phi: \mathbb{A}^\mathbb{N} \rightarrow \mathbb{Z}_p$ be the map of Section ~\ref{sec23}, 
which maps the path set 
$\mathcal{P}$ to the corresponding $p$-adic path set fractal $K = \phi(\mathcal{P})$. 
We have the following Corollary.


\begin{cor} \label{acor1}
 If $\mathcal{P}$ is a path set on the  alphabet $\mathcal{A} =\{0,1,2,\ldots,p-1\}$, then the 
$p$-adic path set fractals $K=\phi(\mathcal{P})$ and $K'=\phi(\mathcal{P}^{(*n)})$ have the same Hausdorff dimension.
\end{cor}
\begin{proof} This follows immediately from (\ref{top-entropy}), Proposition ~\ref{aprop3}, and Proposition ~\ref{pr22a}.
\end{proof}


\begin{rem}\label{remark38}
(1) Corollary \ref{acor1} is useful in computing Hausdorff dimensions  of path sets in our examples.
Let $\mathcal{P}= X(1,4)$ be the Golden Mean Shift, which is also the path set underlying the $3$-adic path set fractal $\mathcal{C}(1,4)$. 
An element of $\mathcal{C}(1,N_k) = \mathcal{C}(1,(10^{k-1}1)_3)$ is any $3$-adic integer consisting of $0$'s and $1$'s and for which no $1$ is followed $k$ digits later by another $1$. Recognizing this property
allows us to see for $N_k = (1 0^{k-1}1)_3 = 3^k +1$
that the path set $X(1,N_k)$ underlying $\mathcal{C}(1,N_k)$ is just $\mathcal{P}^{(*k)}$.
Corollary \ref{acor1} provides another proof of a result in part I (\cite[Theorem 5.5]{AL14c}) asserting that 
$\dim_H(\mathcal{C}(1,N_k)) = \log_3 \phi$, since this now follows from the basic computation $\dim_H(\mathcal{C}(1,4)) = \log_3 \phi$. 
   One may compare
 this argument to the proof given in \cite[Theorem 5.5]{AL14c}. Let $\mathcal{G}$ be the presentation of $\mathcal{C}(1,4)$ given by 
 Algorithm A of \cite{AL14c}. The  algorithm of Proposition ~\ref{aprop1} applied to $k$ and $\mathcal{G}$  and Algorithm A of \cite{AL14c} give isomorphic graph presentations of $\mathcal{C}(1,N_k)$.

(2) In Section ~\ref{secqk} below, we will prove Theorem ~\ref{th111a}, which states that
 \[
\dim_H(\mathcal{C}(1,Q_k)) = \log_3 \phi,
\]
by a similar argument. 
\end{rem}

%
%
%

\section{The  infinite family $P_k= 2 \cdot 3^k +1= (20^{k-1}1)_3$}\label{sec4}

We obtain a relatively complete description of the path set structure for the
family $P_k = 2 \cdot 3^k +1 = (20^{k-1}1)_3$. As a preliminary we review results
for the infinite families $L_k$ and $N_k$ studied in part I (\cite[Section 4]{AL14c}).


%
\subsection{The Family $P_k = (20^{k-1}1)_3=2\cdot 3^k +1$: Path set structure.} \label{sec52}

We study the structure of a path set presentation of the $3$-adic expansions of
elements in $\mathcal{C}(1,P_k)$. The following example gives a path set presentation
for $P_2=19$.

\begin{exmp}\label{example33}
A path  set presentation of the path set $X(1, 19)$ associated to $\mathcal{C}(1,19)$, with $19 = (201)_3$,  is shown in 
 Figure \ref{fig41}. The vertex labeled $0$ is the marked initial vertex.

%
%
%

\begin{figure}[ht]\label{fig41}
 \centering
 \psset{unit=1pt}
 \begin{pspicture}(-100,-165)(100,160)
  \newcommand{\noden}[2]{\node{#1}{#2}{n}}
  \noden{0}{0,110}
  \noden{1}{50,55}
  \noden{10}{50,-55}
  \noden{100}{0,-110}
  \noden{22}{-50,-55}
  \noden{20}{-50,55}
  \noden{21}{0,-25}
  \noden{2}{0,25}
  \bcircle{n0}{0}{0}
  \bcircle{n100}{180}{1}
  \dline{n2}{n21}{1}{0}
  \bline{n0}{n20}{1}
  \bline{n20}{n22}{1}
  \bline{n22}{n100}{1}
  \bline{n100}{n10}{0}
  \aline{n20}{n2}{0}
  \bline{n1}{n0}{0}
  \aline{n10}{n21}{1}
  \bline{n10}{n1}{0}
 \end{pspicture}
 \newline
 \hskip 0.5in {\rm FIGURE 4.1.} Path set presentation of $X(1,19)$. The marked vertex is $0$.
\newline
\newline

\end{figure} 

\newpage

The graph in Figure \ref{fig41}
has adjacency matrix 
\begin{equation*}
\bf{A} = \left(\begin{array}{cccccccc}
1 & 1 & 0 & 0 & 0 & 0 & 0 & 0 \\
0 & 0 & 1 & 1 & 0 & 0 & 0 & 0 \\
0 & 0 & 0 & 0 & 1 & 0 & 0 & 0 \\
0 & 0 & 0 & 0 & 0 & 1 & 0 & 0 \\
0 & 0 & 0 & 0 & 1 & 0 & 1 & 0 \\
0 & 0 & 0 & 1 & 0 & 0 & 0 & 0 \\
0 & 0 & 0 & 0 & 0 & 1 & 0 & 1 \\
1 & 0 & 0 & 0 & 0 & 0 & 0 & 0 \\
\end{array}\right),
\end{equation*}
which has Perron eigenvalue $\beta \approx 1.465571$, so
\[\dim_H(\mathcal{C}(1,19)) = \log_3 \beta \approx 0.347934.\]

An important feature of the graph in Figure \ref{fig41} 
is that it is reducible with  two strongly connected components, 
one component being the $2$ nodes in the middle, and the other the ring of $6$ nodes
around the outside. The (oriented) dependency graph of the strongly connected
components is a  tree with $2$ nodes. 
The Perron eigenvalue $\beta$ of the graph above is associated
with the outer strongly connected component with $6$ nodes. The inner component
has topological entropy $0$.
\end{exmp}

We describe the path set presentation in general.
The vertex labels of the presentation will be described
using the following definition.

\begin{defn}\label{de40}
Classify the labels of the vertices in the graph $G_k$ as numbers $m$ with $0\le m \le 3^k$
whose finite $3$-adic expansions (read right to left) are of types (S1) and (S2) given by:
\begin{enumerate}
\item[(S1)] The expansion $(X)_3$, written with exactly $k$ digits, omits the digit $1$.
\item[(S2)] The $3$-adic expansion  of $m$ contains a single digit $1$,  
   and has the form $( X10^j)_3$ for some $0 \le j \le k$, with $(X10^j)_3$ 
   written with  exactly
   $k$ digits, plus $m = 3^k = (10^k)_3$.
\end{enumerate}
\end{defn}

Note that an (S2) label has initial $3$-adic digits consisting of a string
     of zeros, followed by a $1$.


  \begin{prop} \label{pr49}
   For $P_k= 2 \cdot 3^k +1$ the   path set $X(1, P_k)$ 
   associated to $\mathcal{C}(1, P_k)$ has a presentation $(\sG_k, v_0)$  
   with the following properties.  
   
   (1) The vertices  
   $v_m$ have  labels $m$  consisting of those  $0 \le m \le 3^k$ whose $3$-adic expansion $(m)_3$
   is one of the  two types (S1) and (S2) above. 
     
   (2) The underlying directed graph $G$ of $\sG_k$ has exactly $2^{k+1}$ vertices. 
   
   (3) The reflection map $R(m) = 3^k -m$ which acts on vertex labels of the underlying directed graph $G_k$
    is an  automorphism of $G_k$.  Given any path from $(0)_3$
    to vertex $m$, there is a directed path from vertex $(10^k)_3$ to vertex $3^k -m$ of the
    same length, visiting the set of reflected vertices of the
    original path, and having all the edge labels reversed (exchanging $0$ and $1$).
     \end{prop}
     
     \begin{proof}
     
     The presentation found in this theorem will be that given by the construction of Algorithm A in part I \cite{AL14c}. 
    
     From the proof of Theorem \ref{thm37} we know that a vertex with label $m=3^k$ is reachable
     by a directed path  from 
     vertex $m=0$ and vice-versa.       
     
     We prove the proposition by showing, in order:
     \begin{enumerate}
     \item[(G1)]
     The vertices of $G$ reachable from $v_0$ have labels $0 \le m \le 3^{k}$ which are a 
     subset of  the labels (S1) and (S2).
      \item[(G2)]
      The set of vertex labels $m$ satisfying (S1) or (S2) are
      exchanged under the reflection map $R (m) = 3^k -m$. 
      The set of all  possible $m$ satisfying (S1), respectively  (S2), each  have
     cardinality $2^k$.
     \item[(G3)]
     Each  path emanating from vertex $m=0$ corresponds to a unique path emanating from 
     vertex $m=3^k$
     with the new path having reflected vertex labels and reversed edge labels,
     and vice versa. 
     \item[(G4)]
     The set of all reachable vertices is invariant under
     the reflection map.
     \item[(G5)]
     All  vertices with labels of type (S1) are reachable.  
     \item[(G6)]
     The reflection map on vertices induces  a graph automorphism of $G$ of order $2$
     with no fixed points. Thus $G$ is a double cover of the resulting quotient graph $H$.
          \end{enumerate}
      
      
       To  establish (G1) we proceed by induction on the length $n$ of a shortest path to a given vertex.
      The base case $m=0$ is an (S1) label.
      Following a single  $0$ edge changes a vertex label $(X s)_3$ (with $s=0, 1,$)
      to $(0X)_3$, which maps (S1) labels to (S1) labels and maps (S2) labels to (S2) labels,
      except the case $d=1$ is mapped to an (S1) label. Following a single $1$ edge with vertex
      label $(X s)_3$ (here $s=0, 2$) maps labels having  $s=0$ to $(2X)_3$, which preserves
      the property of being an (S1) label or an (S2) label. For the case $s=2$, which must be
      an (S1) label, rewrite $(Xs)_3 = (Y0 2^j)_3$ for some $j \ge 1$, which is converted to
      $(2Y10^{j-1})_3$, which is an (S2) label. The extreme case $(Xs) = (2^k)_2$ is converted to
      $m=3^k$, in (S2). This completes the induction step.
      
      (G2) There are clearly $2^k$ elements in (S1). The reflection map $R$ acts on elements $m$ of
      (S1) with $m>0$ by replacing each $0$ by $2$ and vice versa, except that the smallest $2$ is converted
      to a $1$, and this is an element of (S2).  The remaining element $m=0$ exchanges with $m=3^k$ 
      which is in (S2).  Conversely elements of (S2) are mapped into elements of (S1), for $m< 3^k$ an  expression
      $10^j$ is converted to $02^j$, and for $m=3^k$ is sent to $m=0$.  Since the reflection map is 
      an involution, it is one to one, so the (S2) labels have
      the same cardinality $2^k$ as (S1) labels.
      
      (G3) This assertion is proved by induction on the length of the path. It is vacuously true at step $0$. 
      For the induction step we  must check that the vertices $m$ and $2^k-m$ have the same 
      number of exit edges, and that the available exit edges have reversed labels in
      the second case. We must also check that following an edge in the two cases leads to
      a pair of reflected vertex labels $m'$ and $3^k -m'$. There are several cases.
      \begin{enumerate}
      \item[Case (1)]
      If $m = (X20^{\ell})_3$ for $\ell >0$ of type (S1), then $3^k - m = (\bar{X}10^{\ell})_3$ is of type (S2).
      Both allow $0$, $1$ exit edges. A $0$ exit edge from $m$ goes to $m'=(0X02^{\ell-1})_3$, and
      a $1$ exit edge for $3^k-m$ goes to $(2\bar{X}10^{\ell-1})_3 = 3^k - m'$. A $1$ exit edge
      from $m$ goes to $m''= (2X20^{\ell -1})_3$, and a $0$ exit edge for $3^k -m $ goes to 
      $(0 \bar{X}10^{\ell -1}) = 3^k - m''$.
      \item[Case (2)]
      If $m = (X02^{\ell})_3$ for $\ell>0$ of type (S1), then $3^k -m = (\bar{X}20^{\ell -1} 1)_3$ is of type (S2).
      Here $m$ allows only a $1$ exit edge, while $3^k-m$ allows only a $0$ exit edge. 
      Under the allowed $1$ exit edge $m$ goes to $m'=(2X10^{\ell -1})_3$ of type  (S2). Under the allowed
      $0$ exit edge $3^k-m$ goes to $(0 \bar{X} 2 0^{\ell -1})_3 = 3^k -m'$ of type (S1).
      \end{enumerate}
      For the two further cases where $m$ is of type (S2), reverse the above. This completes the induction step.

      (G4) By (G3) if a vertex labeled $m$ is reachable from $(0)_3$, then its reflected vertex
      $3^k -m$ is reachable from vertex $3^k$. But vertex $3^k$ is reachable from $(0)_3$
      so $3^k -m$ is reachable from $(0)_3$ as well.
      
      (G5)  We may assume that the  (S1) vertex $m \ne 0$,
      so it has the form $0^{r_0} 2^{r_1} 0^{r_2}\cdots 2^{r_j}$, in which all
      $r_i >0$ except possibly $r_0$ and $r_j$, and $r_0 + r_1 + \cdots + r_j =k$.
      Now it may be realized following a directed path from $(0)_3$
      having successive edge labels
      $1^{r_j}, 0^{r_{j-1}}, 1^{r_{j-2}}, \cdots, 0^{r_0}$. This path is legal,
      because  all intermediate words
      in the path have initial $3$-adic digit $0$ so both edges labeled $0$ and $1$
      exit from that vertex. (The intial word has $k$ initial zeros, and each step can
      decrement the number of leading zeros by  at most $1$). 
      
      (G6) One first checks that each label $m$  in (S1) ending in $0$ corresponds under reflection
      to a label $3^k-m$  in (S2) ending in $0$ and vice versa (since $3$ divides $m$). Each label in (S1) ending 
      in $2$ corresponds under reflection to a label in (S2) ending in $1$; the (S1) label permits only a single
      exit edge with label $1$ and the corresponding (S2) label has a single exit edge labeled $0$.
      Thus at each vertex the reflection automorphism (at the level of vertex labels) preserves 
      the number of edges and reverses their edge labels.  This establishes (G6).
      Moreover the graph G is a double cover of the quotient graph $H$ under the automorphism $R$
      (which has no fixed points).
          \end{proof}

Our next object is to show that  the underlying graph $G_k$ of the path set $X(1,N_k)$
has at least $\lceil \frac{k+1}{2} \rceil$ nested connected components, 
a number which is unbounded as $k \to \infty$.
We establish this  using the following notion  of depth to vertices of  $G_k$.

\begin{defn}\label{de410} 
(1) First we classify  the labels of the vertices in graph $G_k$ as being of types (T1) and (T2) 
as follows:
\begin{enumerate}
\item[(T1)] The $k$-th  $3$-adic digit of $m$ is $0$ or $1$, so $m=(0X)_3$ or $m= (1X)_3$,
with $X$ containing $k-1$ digits,  but excluding the label $m=3^k= (10^k)_3$.
\item
[(T2)] The $k$-th $3$-adic digit of $m$ is $2$, i.e. $m=(2X)_3$, as above, in addition 
including the label $m=3^k= (10^k)_3$.
\end{enumerate}
One may check that there are $2^k$ elements in each set, and that the reflection operation
$R(m)= 3^k -m$ 
sends (T2) labels to (T1) labels and vice versa.\medskip

(2) The {\em depth} of a (T1) label is the number of blocks of consecutive $2$'s
appearing in its $3$-adic expansion. The {\em depth} of  a (T2) label $m$ is the depth of its reflected label
$R(m)$, which is of type (T1).
\end{defn}

Thus $m=0$ and $m=3^k$ are assigned depth $0$. Furthermore all the vertices in the path of
length $2k+2$ studied in the proof of Theorem  \ref{thm37}  are assigned depth $0$,
and they are the complete set of depth $0$ vertices.

The following proposition will establish  that this notion of depth stratifies the strongly
connected components, by showing depth is nondecreasing along each
directed edge. 

  \begin{prop} \label{pr410a}
   For $P_k= 2 \cdot 3^k +1$ the   path set $X(1, P_k)$  has  presentation $(\sG_k, v_0)$  with the following properties.
   
   (1) Each step along an edge in the graph $G_k$  leaves the same or increases the depth of  a vertex.

   (2) For $0 \le j  \le \lfloor k/2\rfloor$
   there are exactly $2 {{k+1}\choose{2j+1}}$ vertices in $\sG_k$ of depth exactly $j$.

   (3)  For each $0 \leq j \leq \lfloor \frac{k}{2} \rfloor$, the vertices of depth $j$ form a strongly connected component of the underlying directed graph $G_k$. Thus, $G_k$ has a sequence of  $1+ \lfloor k/2 \rfloor$  strongly connected components, which are
   nested in a chain.
        \end{prop}

\begin{proof}
The presentation found in this theorem will be that given by the construction of Algorithm A in part I \cite{AL14c}. Some of the notation below only makes sense for $k >3$. We will restrict to these cases, as the result follows for $k=1,2,3$ by direct inspection.
The reversal operation exchanges type (T1) and type (T2) labels. For this to work 
the top $3$-adic digit (the $k$-th digit) must be used, because this is the only digit always reversed
under the reflection map or with $2$ changed to $1$; there is one exception, which is $m=0$ and $m=3^k$, where we
assigned them to (T1) and (T2) directly. The key point is: {\em a label $m$ and its reversal are always at the same
level.}  For the two exceptions $m=0$ and $m=3^k$ this fact had to be checked directly. 
 
 (1) It suffices to check the effect of traversing a single edge in $\sG_k$. 
 The assertion holds for  cases $m=0$ and $m= 3^k$  because
 they both exit to level $0$ vertices.  By the proof of (G3) in Proposition \ref{pr49}, if label $m$ goes to $m'$ by
 edge labeled $s$, then $3^k-m$ goes to $3^k - m'$ by an edge labeled $\bar{s}$. 
 Now the depths of $m$ and $3^k -m$ are  the same, as are those of $m'$ and $3^k -m'$, so it
 suffices to check the effect of following an edge from a vertex of type (T1).  We treat cases.
 
 \begin{enumerate}
      \item[(i)]
      Suppose $m = (0X0)_3$ of type (T1) has depth $d$, thus $X$ contains   
      $d$  blocks of consecutive $2$'s. Following a $0$ edge goes to $m'= (00X)_3$,
      also (T1) of depth $d$. 
      \item[(ii)]
      Suppose $m = (0X0)_3$ of type (T1) has depth $d$, thus it has 
      $d$  blocks of consecutive $2$'s. Following a $1$ edge goes to $m'= (20X)_3$,
      now (T2), of depth same as $3^k-m'$. Now  $X= X'2 0^{\ell}$ with $\ell \ge 0$ or
       $X =0^{\ell}$. In the first case
      $3^k - m' = (02 \bar{X'} 10^{\ell})_3$  If $X' =0X'' 0$, then it has $d-1$ blocks of $2's$, but its
      reversal $\bar{X}$ has $d$ blocks.  If $X'=2X''0$ then it has $d-1$ blocks of $2's$, as does
      its reversal, but the $02$ at front creates another block. If $X'=0X'2$ then it has $d$ blocks
      of $2$'s, as does its reversal. Finally if $X'=2X'2$ then it has $d$ blocks of $2$'s, its reversal
      has $d-1$ blocks, but the $02$ at front creats another blocks.  In all cases the depth cannot 
      decrease.
     \item[(iii)]
       Suppose $m =(0X02^{\ell})_3$ with $\ell>0$ of type (T1) has depth $d$.  Now can only follow a $1$ edge,
       go to $m'= (2 0X 10^{\ell -1})_3$ is of type (T2). This has same depth as 
       $3^k-m' = (02 \bar{X} 2 0^{\ell -1})_3$. Now $X$ has $d-1$ blocks of $2$'s. If it is of form $0 X''0$ then
       reversal increases number of blocks of $2$'s in it by $1$, compensating exactly for the lost $2$ block at the right end of the label, so the depth is still $d$. If of form $2X''0$ or $ 0 X''2$
       then reversal leaves $d-1$ blocks of $2$'s but get one extra block from either $2$ before or after, so the depth is still $d$.
       If of form $2 X'' 2$ then reversal leaves $d-2$ blocks of $2$'s but now gain two extra blocks from
       the $2$ before and after, so the depth is still $d$.
      \end{enumerate}
      In all cases of  a type (T1) vertex a step leaves depth the same or increases it by $1$.

(2) Let $k$ be fixed. The 
result is true for $j=0$ by the construction in Theorem \ref{thm37}, where there are $2k+2 = 2{{k+1}\choose{1}}$
vertices of depth $0$, and this component is strongly connected. 

For $j \ge 1$ it  suffices to count the number of labels of type (T1) at depth $j$ and then double it.
For $j \ge 1$ the  number of labels of type $(T1)$ at depth $j$ consist of all labels of form
$(0^{k_1}2^{\ell_1} 0^{k_2} 2^{\ell_2} \cdots 0^{k_j} 2^{\ell_j} 0^{k_{j+1}}X)_3$
with final block $X= \emptyset$ (set $k_{j+2}=0$) or $X=(1 0^{k_{j+2} -1})$ (the latter requires $k_{j+2} \ge 1$).
Since labels have length $k$ the  exponents necessarily satisfy
\[
k_1+ \cdots  + k_{j+1} +k_{j+2} + \ell_1 + \cdots + \ell_j =k, ~~ \, k_i, \ell_i >0 \, \mbox{for} \, 1 \le i \le j; k_{j+1}, k_{j+2} \ge 0.
\]
There are ${{k}\choose{2j}}$ solutions of depth $j$ type $(T1)$ with $X$ not containing a $1$; this follows since
there are $k$ symblols in a label and  we mark the final elements of each
$0^{k_i}$ and $2^{k_i}$ with an asterisk for $1 \le i \le j$ to uniquely determine a depth $j$ label with $X =\emptyset$.
There are  ${{k}\choose{2j+1}}$ solutions of depth $j$ type $(T1)$ with $X$ containing a $1$; here we add
an additional asterisk marking the $1$, which unqiuely specifies the label,
so we have the number of ways of inserting $2j+1$ asterisks.   Thus the  number of $(T1)$
labels  of depth $j$ is  ${{k+1}\choose{2j+1}}$,
and (2) follows.

(3) First, we show that it is possible to reach a vertex of each depth $0 \leq j \leq \lfloor k/2 \rfloor$. Starting from $m=0$ following paths with labels $(10)^j$ for $1 \le j \le\lfloor k/2\rfloor$,
one arrives at vertices $m_{2j} := ( (02)^j 0^{k-2j})_3$, and $m_{2j}$ is a type (T1) label of depth $j$.
These are legal paths since all the intermediate vertex $m_j$ labels (for $1 \le j \le m-1$) have initial $3$-adic
digit $0$. We have produced a path with vertices of depth $0, 1, 2, ..., \lfloor k/2\rfloor$, which
guarantees the existence of at least one  sequence of distinct strongly connected components of length
$1+ \lfloor k/2\rfloor$ which are nested in a chain. 

Next, we show that the subgraph of $G_k$ consisting of those vertices of depth $j$ is strongly connected for each $0 \leq j \leq \lfloor k/2 \rfloor$. At depth $d=0$, beginning at the vertext labeled $0$ and traversing a path with label $1^{k+1} 0^{k+1}$ gives a loop at the $0$-vertex that passes through each other vertex of depth $0$, so the subgraph of depth $0$ vertices is strongly connected.. Below, we restrict attention to depths $d \geq 1$, and some statements below only apply in those cases. Recall also that we are restricting attention to $k > 3$, as smaller cases can be checked by hand.

We need to show, firstly, that from any vertex it is always possible to traverse an edge that leaves the depth unchanged. By the proof of (G3) in Proposition ~\ref{pr49} and the discussion in the first paragraph of (1) above, it suffices to verify this for vertices of type (T1). Let $m$ be the label of a vertex of depth $d$ and type (T1). Then either $m = (0X0)_3$, in which case we may follow an edge labeled $0$ to arrive at a vertex labeled $(00X)_3$ that also has depth $d$, or else $m = (0X02^l)_3$ for some $l > 0$. In the latter case, we may follow an edge labeled $1$ to a vertex labeled $(20X10^{l-1})_3$, and the discussion in (iii) above shows that this vertex also has depth $d$. In any case, we can always traverse an edge that will leave the depth unchanged.

Among depth $d$ labels, the minimal such label is $m_{min} = ((20)^{d-1}2)_3$. In order to show that the set of depth $d$ vertices is a strongly connected subgraph of $\mathcal{G}_k$, it suffices to show that it is always possible, beginning at any vertex of depth $d$, to traverse paths both \emph{forwards} to $m_{min}$ and \emph{backwards} to the same vertex (that is, contrary to the ordinary direction that arrows are traversed; this will show that there is a path forwards from $m_{min}$ to the desired vertex). This will follow if we can show that:
\begin{enumerate}[(A)]
\item For any depth $d$ vertex with non-minimal label $m$, it is always possible to follow a path, staying at depth $d$, to another vertex with label $m' < m$.
\item For any depth $d$ vertex, it is possible to follow edges $backwards$ until we reach a vertex where each block of $2$'s has length exactly $1$.
\item For any depth $d$ vertex with a label where each block of $2$'s has length exactly $1$, it is possible to reach $m_{min}$ by going backwards.
\end{enumerate}

(A) Suppose now we are at a depth $d$ vertex with label $m$ of type (T1). Then either $m$ is of the form $(0X0)_3$, or else $m$ is of the form $(0X02^l)_3$ for some $l >0$. If $m = (0X0)_3$, then we may traverse an edge labeled $0$ to arrive at an edge labeled $m' = (0X)_3 < m$, and $m'$ is also at depth $d$. Now suppose instead that $m = (0X02^l)_3$. Then we must traverse next an edge labeled $1$ to the vertex with label $m' = (20X10^{l-1})_3 > m$. By the argument of (iii) above, this vertex also has depth $d$. From here, we may traverse $l$ consecutive edges labeled $0$ to arrive at a vertex labeled  $m'' = (20X)_3$, whose depth is also $d$. If the right-most digit of $X$ is not a $2$, we may continue to traverse edges labeled $0$ until we arrive at a vertex $m''' = (20Y)_3$ where the right-most digit of $Y$ is a $2$, and the length $|Y| \leq |X|$, or else at the vertex $m^{(4)} = (2)_3$ if $X$ is the empty string. In the latter case, we are at depth $d=1$ and $m^{(4)} = (2)_3 = m_{min}$ is already the minimal label. Suppose we are in the former case, and we have arrived at $m''' = (20Y)_3$. But for any $l \geq 1$, we necessarily have $m ''' =(20Y)_3 \leq (X02^l)_3 =m$, with equality if and only if $X = Y$, $l=1$, and $m = m' = (20)^{d-1}2 = m_{min}$. Thus, in any case, we may always traverse a path, remaining at depth $d$, to arrive at a vertex whose label is less than $m$.

What if our initial vertex is of type (T2)? Then, $m$ is either of the form $10^k$, in which case, we simply follow edges labeled $1$ until we reach the vertex labeled $0$, or we have something of the form $2X$, where $X$ has $k-1$ digits. In this case, if $X$ terminates in $10^l$, we can immediately follow a vertex $0$, without dropping depths, to $m'$ of form $(T1)$, where of course $m'< m$. Otherwise, we have $2Y20^l$, where we follow $l + 1$ edges of label $1$; the first $l$ bring us to $2Z2$, and the $(l+1)$st edge takes us to a (T2) vertex that terminates in $10^n$, which is a case already covered.

This proves (A).

To see (B), we will devise an algorithm (call it Algorithm (B)).
\begin{enumerate}[(i)]
\item If we are at $2X10^l$ then we follow a vertex labeled 1
backwards to vertex $X02^{l+1}$. (This does not drop depth, as a block
of consecutive $2$'s necessarily transforms into another block of
consecutive $2$'s).
\item If we are at $0^lX$, where $l > 1$, or we are at $0^lY10^n$, where
$l>0$, we follow a vertex labelled $0$ to $0^{l-1}X$ or $0^{l-1}Y10^{n+1}$.
\item If we are at $02X$, and $X$ omits the digit $1$, we follow an edge
labeled $0$ back to $2X1$. Notice that this avoids dropping depth.
\item If we are $2X$, where $X$ omits the digit $1$, we follow an edge
labeled $1$ back to $X0$.
\end{enumerate}

The crux is step (iii); following the notation of that step, we will then be at $2X1$, with no $0$s after the $1$. We then apply case (i), reaching $X02$. Any other $2$'s that appeared in the block at the far left will be transformed into $0$'s on the far right by the application of step $(iv)$, while the other blocks will merely be shifted.

Thereby, by repeated application of this algorithm, all of the blocks will be transformed into single-digit blocks after at most $k$ iterations. This concludes (B). For an illustration at depth 2, see the column labeled ``Step (B)" in Table 4.1.

Finally, for (C), notice that, for the type of vertex we are interested in, repeated application of Algorithm (B) simply "scrolls through" the label, with the blocks of $2$'s shifting left, always preserving the same cyclic order, with the same gaps of $0$'s between them (unless a $1$ is present) between them. In the case of the illustration of Table 4.1, see the column labeled ``Step (C)-1" of that table.

So, for (C), apply Algorithm (B) until we are at $0^lX2$ where $l > 1$ (if this is strictly impossible, then simply "scroll" until we are at $(02)^{k/2}$, and at this depth, that is the minimal vertex). Then, break the pattern and go to $0^lX21$. Then, continue to apply Algorithm (B) until we return to a vertex where all of the blocks of $2$'s have length $1$.

Essentially, we will generate a long block of $2$'s instead of the block of $0$'s we currently have, which won't have such a large gap; see the column labeled ``Step (C)-2" in Table 4.1.

One such procedure transforms a block of $0$'s of arbitrary length into a block of length $1$.

Repeat this procedure untill all of the blocks of $0$'s (except for 1) have length $1$, and then use Algorithm (B) until we reach the minimal vertex. This completes (3). Continuing with our simple example, see the column labeled ``Step (C)-3" in Table 4.1.

%
%

\begin{minipage}{\linewidth}
\begin{center}
\begin{tabular}{ |r|r|r|r| }
\hline
\mbox{Step (B)} &  \mbox{Step (C)-1} & \mbox{Step (C)-2} & \mbox{Step (C)-3} \\ 
\hline
22022022 & 0020002 & 0020002 & 0002020 \\
20220220 & 0200020 & 0200021 & 0020200 \\
02202200 & 2000201 & 2000210 & 0202000 \\
22002201 & 0002002 & 0002022 & 2020001 \\
20022002 & 0020020 & 0020220 & 0200002 \\
00220020 & 0200200 & 0202200 & 2000021 \\
02200200 & 2002001 & 2022001 & 0000202 \\
22002001 & 0020002 & 0220002 &  \\
20020002 &  & 2200021 &  \\
00200020 &  & 2000202  &  \\
 &  & 0002020 & \\
 \hline
\end{tabular} \par
\bigskip
\hskip 0.5in {\rm TABLE 4.1.} Example of algorithm for proof of Proposition 4.3(3). 
\newline
\newline
\end{center}
\end{minipage}
\end{proof}


\begin{rem} \label{rem45}
(1) Proposition \ref{pr410a}  counts the number of vertices at each depth,
giving a recursion to compute them.
Table 4.2 below  gives values  for $1 \le k \le 9$.\\

%
%

\begin{minipage}{\linewidth}
\begin{center}
\begin{tabular}{ |l l |r|r| r | r | r | r | r | }
\hline
  $~~~$ &\mbox{Depth=} &  $0$ &  $1$ & $2$ & $3$ & $4$ \\ 
\hline
 $P_1 =7$ & &  $4$ & $$  & $$ & $$ & $$ \\
$P_2 =19$ & &  $6$ & $2$  & $$ & $$ & $$ \\
$P_3= 55$ & & $8$ & $8$  & $$ & $$ & $$ \\
$P_4= 163$ &&  $10$ & $20$  & $2$ & $$ & $$ \\
$P_5= 487$ & & $12$ & $40$  & $12$ & $$ & $$ \\
$P_6=1459$ & &  $14$ & $70$  & $42$ & $2$ & $$ \\
$P_7=4375$ & &$16$ & $112$  & $112$ & $16$ & $$ \\
$P_8=13123$ &&  $18$ & $168$  & $252$ & $72$ & $2$ \\
$P_9=39367$ & & $20$ & $240$  & $504$ & $240$ & $20$ \\ \hline
\end{tabular} \par
\bigskip
\hskip 0.5in {\rm TABLE 4.2.}  Number of vertices at given depth in graph $\sG_k$ for $X(1, P_k)$. 
\newline
\newline
\end{center}
\end{minipage}

(2)
Proposition \ref{pr410a} says that the graph $X(1, P_k)$ has a ``Matryoshka doll" structure of
a single set of nested strongly connected components, one at each depth $0 \leq j \leq \lfloor k/2 \rfloor$. 

(3) 
The proof of Proposition \ref{pr410a} exploits repeatedly the symmetry of the graph $G_k$ exhibited by the partitioning of vertices into types (T1) and (T2).
\end{rem}

%
%
%
%
%
%
\subsection{The Family $P_k = (20^{k-1}1)_3=2\cdot 3^k +1$: Hausdorff dimension.} \label{sec53}

Data on the  Hausdorff dimensions of the first few of the sets $\mathcal{C}(1, P_k)$ 
were obtained by computer calculation of the maximum eigenvalue of  the 
adjacency matrix of the
graph $X(1, P_k)$ and presented in Section \ref{sec21}. The data contained oscillations
and other features which we discuss in Remark \ref{rem47} below.

We now lower bound the Hausdorff dimension of $\mathcal{C}(1, P_k)$ as $k \to \infty$.
Theorem \ref{th109c} gives  both an asymptotic limiting result and a lower bound because it may be that the 
Hausdorff dimensions continue to oscillate for large $k$.

%
%
\begin{proof}[Proof of Theorem \ref{th109c}]
Let $a = \lfloor \frac{k}{4} \rfloor$ and let $b \in \{0,1,2,3\}$ be congruent to $k$ mod $4$, so that $k = 4a + b$. Let $S \subset \mathcal{A}^\mathbb{N} = \{0,1,2\}^\mathbb{N}$ be given by 
\begin{equation} \label{defS} S = \{(1100)^a0^b((1x00)^a0^b(1000)^{a-1}1000^b)^\infty \in \mathcal{A}^\mathbb{N} | x \in \{0,1\} \text{ may vary}\}.
\end{equation}
What we will show is that $S \subset X(1,P_k)$. Since elements of $S$, after the fixed initial string $(1100)^a 0^b$, consists of symbol sequences of length $2k-1$ with $2k-1-a$ fixed digits and $a$ digits which may be either $0$ or $1$, it follows that 
\[H_{top} (S) = \frac{a}{2k-1} \log_3(2) = \frac{\lfloor \frac{k}{4} \rfloor}{2k-1} \log_3(2).
\]
The two inequalities of the theorem, that
\[
\liminf_{k \to \infty} \dim_{H} \mathcal{C} (1, P_k) \ge  \frac{1}{8} \log_3(2),
\]
and, for all $k$, 
\[\dim_H(\mathcal{C}(1,P_k)) \geq \frac{1}{13} \log_3(2),
\]
then will follow immediately.

To prove that $S \subset X(1,P_k)$, we will trace out paths on the graph presentation of $\mathcal{C}(1,P_k)$ given by Algorithm A of \cite{AL14c} whose edge labels give the  elements of $S$. First, note that if we begin with an edge labeled $1$ from the $0$-vertex, we arrive at the vertex with label $20^{k-1}$. This means that our next $k-1$ vertices may be either $0$ or $1$ freely. Each edge $0$ appends a $0$ to the front of the vertex label and removes the last digit, and each edge $1$ appends a $2$ to the front of the vertex label and removes the last digit. From these observations, we see that there is in fact a sequence of edges with label $(1100)^a0^b$, and having traversed these edges we arrive at a vertex labeled $0^b(0022)^a$. Call this vertex $v$. 

We will now show that we may traverse a sequence of edges with label \newline $(1x00)^a0^b(1000)^{a-1}1000^b$ initiating at $v$ for $x=0$ and $x=1$, and that such a path also terminates at $v$. The result will follow. Now since the label of $v$ ends in $2$, the only out edge is indeed labeled $1$, and this takes us to a vertex labeled $20^b(0022)^{a-1}010$. The next edge label $x$ may then be either $0$ of $1$, terminating in a vertex labeled $[2x]20^b(0022)^{a-1}01$, where $[2x]$ is a digit given by the product of $2$ and $x$. From this vertex we may traverse two subsequent edges each labeled $0$, and the target vertex is $00[2x]20^b(0022)^{a-1}$. It is easy to see that we may repeat this process, traversing edges labeled $(1x00)$ $a$ times and ultimately terminating at a vertex labeled $(00[2x]2)^a0^b$. Traversing then $b$ edges labeled $0$ gets us to the vertex labeled $0^b(00[2x]2)^a$. We may then traverse edges labeled $(1000)^{a-1}1000^b$ to arrive back at the vertex $v$ labeled $0^b(0022)^a$. This completes the proof.
\end{proof}

\begin{rem}\label{rem47}
We speculate on the behavior of the Hausdorff dimension function $\mathcal{C}(1, P_k)$ as  a function of $k$.
We believe the following might be true.
\begin{enumerate}
\item[(1)]
Fixing level $j$ and varying $k$ the topological entropy of the strongly connected component at depth
$j$ stay at value $0$ until $k \ge 2j-2$, then increas monotonically to a maximum and then decrease monotonically
thereafter. 
\item[(2)]
The ``champion" depth $j$ with maximal topological entropy is a nondecreasing function of $k$.
\end{enumerate} 
Speculations (1) and (2) are suggested by  analogy with the behavior of the number of vertices at depth $j$ as a function of $k$,
given in Table 4.1, which have both these properties.
\end{rem}
%
%
%

\subsection{Hausdorff dimension bounds for  $\mathcal{C}(1, P_{k_1}, ..., P_{k_n})$  }\label{sec54}

The path set structures of the members of the infinite  family $P_k$  are compatible with
each other, as a function of $k$, so that the associated $\mathcal{C}(1, P_{k_1}, ..., P_{k_n})$
all have positive Hausdorff dimension.  We relate these Hausdorff dimensions to those of  the infinite family 
$L_k = (1^k)_3 = \frac{1}{2} (3^{k+1} -1)$
treated by the first and third authors in \cite{AL14c} and reviewed in Appendix A (Section ~\ref{secA0}).

\begin{thm} \label{thm414} 
For the family $P_k= 2 \cdot 3^{k} +1 = (20^{k-1}1)_3$, 
and  $0 \leq k_1 < \ldots < k_n$, the graph $\mathcal{G}$ presenting the path set $X(1, P_{k_1}, ..., P_{k_n})$ underlying
 $\mathcal{C}(1,P_{k_1},\ldots,P_{k_n})$ contains a double covering of the 
 underlying directed graph $G_{(1^{k_n+2})_3}$ presenting the path set $X (1, L_{k_n+1})$ underlying 
$\mathcal{C}(1,L_{k_n+1})$. Consequently
\begin{equation} 
\dim_H(\mathcal{C}(1,P_{k_1},\ldots,P_{k_n})) \geq \dim_H(\mathcal{C}(1,L_{k_n+2})).
\end{equation}
\end{thm}

\begin{proof} The graphs under consideration are the graphs given by Algorithm A of \cite{AL14c}. Since the underlying graph $G_k$ of the path set presentation $(\sG_k, v_0)$ 
of the path set $X(1, P_k)$ 
 contains a double covering of the underlying graph $G_{k+1}^{'}$ of the path set presentation of
 $X(1, L_{k+1})$, and
\[ \mathcal{G}_{(1^{k_1+2})_3} \star \cdots \star \mathcal{G}_{(1^{k_n+2})_3} \cong  \mathcal{G}_{(1^{k_n+2})_3},\]
the proposition follows from Theorem ~\ref{thm37} in Appendix B.

Note that this directed graph covering  is not a covering at the level of path sets, because the path labels 
on the two graphs differ.
\end{proof}

Theorem \ref{thm414}  shows that there exist an arbitrarily large number of different values $M_j$,
each  having a $2$ in their ternary expansion,
such that $\dim_{H}(\mathcal{C}(1, M_1, M_2, ..., M_n))>0$.

%
%

\section{The infinite family $Q_k = 3^{2k}-3^k+1 = (2^k0^{k-1}1)_3$} \label{secqk}

Let $Q_k = 3^{2k}-3^k+1 = (2^k0^{k-1}1)_3$. We will prove Theorem ~\ref{thqkstruct}, 
which describes the structure of a graph presentation $\mathcal{G}_k$ of $\mathcal{C}(1,Q_k)$. 
We then use this description to  prove Theorem ~\ref{th111a}, which computes the Hausdorff dimension of $\mathcal{C}(1,Q_k)$.

%
%
\subsection{The Family $Q_k = (2^k0^{k-1}1)_3 = 3^{2k}-3^k +1$: Path set structure}

First, let us give an example. The following example gives a path set presentation for $Q_2 = 73$.

%
%
%

\begin{exmp}\label{examp51} 
A path set presentation of $X(1,73)$, with $73=(2201)_3 $, is shown in Figure ~\ref{fig51}. The vertex labeled $0$ is the marked initial vertex.

%
%
%

\begin{figure}[ht]\label{fig51}
 \centering
 \psset{unit=1pt}
 \begin{pspicture}(-250,-250)(250,250)
  \newcommand{\noden}[2]{\node{#1}{#2}{n}}
  \noden{0}{-75,190}
  \noden{220}{-200,150}
  \noden{1012}{-240,0}
  \noden{1022}{-200,-150}
  \noden{1100}{-75,-190}
  \noden{110}{50,-150}
  \noden{11}{90,0}
  \noden{1}{50,150}
  \noden{22}{-150,75}
  \noden{1020}{-150,-75}
  \noden{102}{-75,-125}
  \noden{1001}{0,-75}
  \noden{100}{-25,0}
  \noden{10}{0,75}
  \noden{221}{-75,125}
  \noden{1000}{-125,0}

  \bcircle{n0}{0}{0}
  \bline{n0}{n220}{1}
  \bline{n220}{n1012}{1}
  \aline{n220}{n22}{0}
  \bline{n1012}{n1022}{1}
  \bline{n1022}{n1100}{1}
  \bcircle{n1100}{180}{1}
  \bline{n1100}{n110}{0}
  \bline{n110}{n11}{0}
  \aline{n110}{n1001}{1}
  \bline{n11}{n1}{0}
  \bline{n1}{n0}{0}
  \bline{n1020}{n102}{0}
  \bline{n102}{n1001}{1}
  \bline{n1020}{n1022}{1}
  \bline{n1001}{n100}{0}
  \bline{n100}{n10}{0}
  \bline{n10}{n1}{0}
  \bline{n10}{n221}{1}
  \bline{n221}{n22}{0}
  \bline{n22}{n1000}{1}
  \bline{n1000}{n1020}{1}
  \dline{n1000}{n100}{0}{1}
 \end{pspicture}
 \newline
 \hskip 0.5in {\rm FIGURE 5.1} Path set presentation of $X(1,73)$. The marked vertex is $0$.
\newline
\newline

\end{figure} 

The graph in Figure ~\ref{fig51} has adjacency matrix 
\begin{equation*}
\bf{A} = \left(\begin{array}{cccccccccccccccc}
1 & 1 & 0 & 0 & 0 & 0 & 0 & 0 & 0 & 0 & 0 & 0 & 0 & 0 & 0 & 0 \\
0 & 0 & 1 & 1 & 0 & 0 & 0 & 0 & 0 & 0 & 0 & 0 & 0 & 0 & 0 & 0 \\
0 & 0 & 0 & 0 & 1 & 0 & 0 & 0 & 0 & 0 & 0 & 0 & 0 & 0 & 0 & 0 \\
0 & 0 & 0 & 0 & 0 & 1 & 0 & 0 & 0 & 0 & 0 & 0 & 0 & 0 & 0 & 0 \\
0 & 0 & 0 & 0 & 0 & 0 & 1 & 0 & 0 & 0 & 0 & 0 & 0 & 0 & 0 & 0 \\
0 & 0 & 0 & 0 & 0 & 0 & 0 & 1 & 1 & 0 & 0 & 0 & 0 & 0 & 0 & 0 \\
0 & 0 & 0 & 0 & 0 & 0 & 1 & 0 & 0 & 1 & 0 & 0 & 0 & 0 & 0 & 0 \\
0 & 0 & 0 & 0 & 1 & 0 & 0 & 0 & 0 & 0 & 1 & 0 & 0 & 0 & 0 & 0 \\
0 & 0 & 0 & 0 & 0 & 1 & 0 & 0 & 0 & 0 & 0 & 1 & 0 & 0 & 0 & 0 \\
0 & 0 & 0 & 0 & 0 & 0 & 0 & 0 & 0 & 0 & 0 & 0 & 1 & 1 & 0 & 0 \\
0 & 0 & 0 & 0 & 0 & 0 & 0 & 0 & 0 & 0 & 0 & 0 & 1 & 0 & 0 & 0 \\
0 & 0 & 0 & 0 & 0 & 0 & 0 & 0 & 0 & 0 & 0 & 0 & 0 & 0 & 1 & 1 \\
0 & 0 & 0 & 0 & 0 & 0 & 0 & 0 & 1 & 0 & 0 & 0 & 0 & 0 & 0 & 0 \\
0 & 0 & 0 & 0 & 0 & 0 & 0 & 0 & 0 & 0 & 0 & 0 & 0 & 0 & 0 & 1 \\
0 & 0 & 0 & 1 & 0 & 0 & 0 & 0 & 0 & 0 & 0 & 0 & 0 & 0 & 0 & 0 \\
1 & 0 & 0 & 0 & 0 & 0 & 0 & 0 & 0 & 0 & 0 & 0 & 0 & 0 & 0 & 0 \\
\end{array}\right),
\end{equation*}
which has Perron eigenvalue $\beta = \frac{1+ \sqrt{5}}{2}$, so
\[\dim_H(\mathcal{C}(1,73)) = \log_3 \left( \frac{1+\sqrt{5}}{2}\right) \approx 0.438108.\]

\end{exmp}

We describe the path set presentation in general. Theorem ~\ref{thqkstruct} will follow easily from the following result, which makes use of the concepts developed in Section ~\ref{sec34}.

%
%
%

\begin{prop} \label{qkprop1} 
Let $\mathcal{P} = X(1,7)$ be the path set underlying $\mathcal{C}(1,7)$, and let $\mathcal{Q} = X(1,Q_k)$ be the path set underlying $\mathcal{C}(1,Q_k)$. Then $\mathcal{Q}$ is the interleaved path set 
\begin{equation}
\mathcal{Q} = \mathcal{P}^{(*k)}.
\end{equation}
\end{prop}
\begin{proof} For convenience, we recall that $\mathcal{P} = X_{\mathcal{G}}(0)$ for the graph $\mathcal{G}$ in Figure ~\ref{figqk1}. This is the graph given by the Algorithm A of \cite{AL14c}.

\begin{figure}[ht]\label{figqk1}
	\centering
	\psset{unit=1pt}
	\begin{pspicture}(-80,-50)(80,150)
		\newcommand{\noden}[2]{\node{#1}{#2}{n}}
		\noden{0}{0,100}
		\noden{1}{50,50}
		\noden{2}{-50,50}
		\noden{10}{0,0}
		\bcircle{n0}{0}{0}
		\bcircle{n10}{180}{1}
		\bline{n0}{n2}{1}
		\bline{n2}{n10}{1}
		\bline{n10}{n1}{0}
		\bline{n1}{n0}{0}
	\end{pspicture}
	\newline
\hskip 0.5in {\rm FIGURE 5.2.} Path set presentation of $X(1,7)$. The marked vertex is $0$.
\newline
\newline
\end{figure}

Let $(\mathcal{H},v_0)$ be the graph presentation of $\mathcal{Q}$ given by the same algorithm. An element of $\mathcal{P}$ may begin with either a $0$ or a $1$, while an element $(x_i)_{i=0}^\infty$ of $\mathcal{Q}$ may begin with any sequence $x_0x_1 \cdots x_{k-1}$ of $0$'s and $1$'s, since $Q_k$ terminates in $0^{k-1}1$. Thus, the initial $k$-blocks of $\mathcal{Q}$ are precisely the same as the initial $k$-blocks of 
the interleaved path set $\mathcal{P}^{(*k)}$. 

To show that $\mathcal{Q} = \mathcal{P}^{(*k)}$ we
just need to check that for each $0 \leq j \leq k-1$, the admissible strings $x_j x_{j+k} x_{j+2k} \cdots$ of $j ~(\bmod k)$ digits of elements of $\mathcal{Q}$ are precisely the elements of $\mathcal{P}$. 
We proceed by induction on $j \ge 0$, the observation above completing the base case $j=0$.
Inductively, assume none of the digits $x_r$ for $r \equiv l ~(\bmod k)$ with $l < j$ can restrict the admissible values for the digits $x_{j + nk}$ 
for $n \geq 0$. We mean here that whether $x_r = 0$ or $x_r=1$ has no effect on the last digit of the vertex label in $\mathcal{H}$ 
arrived at from a path labeled $x_0x_1\cdots x_{j + nk}$ originating at $v_0$. The base case, $j=0$, is satisfied trivially. 
Then we can without loss of generality assume $x_i = 0$ for all $0 \leq i < j$. For now, we will also assume that $x_r = 0$
 for all $r \not\equiv j ~(\bmod k)$. This assumption is not as restrictive as it seems since, as we will show, the $j ~(\bmod k)$
  digits do not effect the available choices for digits of other modular classes. Now since $Q_k = 2^k 0^{k-1}1$, whether $x_j$ is 
  $0$ or $1$ has no effect on the digits $x_{j+1}, x_{j+2}, \ldots, x_{j+k-1}$. If $x_j = 0$, then $x_{j+k}$ may also be either $0$ or $1$.
   If $x_{j+mk}$ is $0$ for all $m < n$, then also $x_{j + nk}$ may be either $0$ or $1$, and those $x_r$ for $r < j +nk$, 
   $r \not\equiv j ~(\bmod k)$ are unrestricted. On the other hand, suppose there is an $n \geq 0$ such that $x_{j + mk} = 0$ for all
    $m < n$ and $x_{j + nk} = 1$. Again, the labels $x_r$ for $r < j + (n+1)k$, $r \not\equiv j ~(\bmod k)$ are unrestricted. 
    However, $x_{j + (n+1)k}$ must now be a $1$. Now the label of the vertex we are at, having traversed the path labeled 
    $x_0x_1 \cdots x_{j+(n+1)k}$ from $v_0$, has label $10^{2k-1}$. Thus the digits $x_{j+(n+1)k +1},x_{j + (n+1)k +2}, \cdots x_{j+(n+3)k-1}$ 
    are unrestricted. However, if the digit $x_{j + (n+2)k}$ is a $1$, then the vertex  at the end of the path labeled $x_0x_1 \cdots x_{j+(n+2)k}$
     has label $10^{2k-1}$, so the vertices after $x_{j + (n+2)k}$ are restricted or unrestricted in precisely the same way as those after
      $x_{j + (n+1)k}$. If on the other hand $x_{j+(n+2)k} = 0$, then the terminal vertex has label $10^{k-2}$. 
      Thus, the label of the vertex after $j + (n+3)k-1$ steps in this case is $1$, hence in this case $x_{j+(n+3)k}$ must be $0$. 
      The resulting terminal vertex label is $0$. In either case, the digits, $x_{j+ (n+3)k+1}, x_{j+(n+3)k+2}, x_{j + (n+4)k-1}$ are unrestricted. 
      For the $(j + (n+4)k)$th step we either begin at vertex $0$ or at vertex $10^{k-1}$, which cases have already been considered.

Thus, we have shown that the digits $x_{j+nk}$ place no restrictions on any digits from the other modular classes, and, furthermore,
 we have described the restrictions that $x_{j+nk}$ place on $x_{j + mk}$ for $m > n$. Inspecting this description shows that the 
 admissible digits $x_j x_{j+k} x_{j+2k}$ are precisely the edge labels of the infinite walks in $\mathcal{G}$ originating at the vertex $0$ 
 in Figure ~\ref{figqk1}. These are precisely the elements of $\mathcal{P}$, so $\mathcal{Q} = \mathcal{P}^{(*k)}$.
\end{proof}

Let $\mathcal{G}$ be the graph of Figure ~\ref{figqk1}. The presentation for $Q_k$ given by Proposition ~\ref{aprop1} applied to $k$ and $\mathcal{G}$ is isomorphic to that given by Algorithm A of \cite{AL14c}. We are now ready to prove Theorem ~\ref{thqkstruct}.

\begin{proof}[Proof of Theorem ~\ref{thqkstruct}.] Let $(\mathcal{G}_k,v_0)$ be the presentation of $\mathcal{Q} = X(1, Q_k)$ constructed by applying the algorithm of Proposition ~\ref{aprop1} to the presentation $\mathcal{G}$ of $X(1, 7)$. Since the graph $\mathcal{G}$ used in this construction has $4$ vertices and $6$ edges, it follows by Proposition ~\ref{aprop1} that $\mathcal{G}_k$ has $4^k$ vertices and $6 \cdot 4^{k-1}$ edges. Moreover, since $\mathcal{G}$ is strongly connected, so is $\mathcal{G}_k$, by Remark ~\ref{rem35}. This proves the theorem. 
\end{proof}

%
%
%

\subsection{The family $Q_k = (2^k0^{k-1}1)_3 = 3^{2k}-3^k +1$: Hausdorff dimension}

We have shown that 
\begin{equation} 
X(1,Q_k) = X(1,7)^{(*k)},
\end{equation}
is given by an interleaving construction. 
Using the results of Section ~\ref{sec34},  
it is now a simple matter to prove Theorem ~\ref{th111a}. 

\begin{proof}[Proof of Theorem ~\ref{th111a}] We are trying to show that
\[
\dim_H(\mathcal{C}(1,Q_k)) = \log_3 \phi.
\]
The result follows by Proposition ~\ref{qkprop1} and by application of the
interleaving result given in  Corollary ~\ref{acor1}, since 
\[
\dim_H(\mathcal{C}(1,7)) = \log_3 \phi,
\]
as is easily computed, and Corollary ~\ref{acor1} shows that the interleaving operation $(\cdot)^{(*k)}$ preserves the topological entropy of 
the input path set. 
\end{proof}

%
%
%

\section{Bounds on Hausdorff dimensions by numbers of ternary digits}\label{sec6}

We study properties of the Hausdorff dimension constants $\alpha_n$.

%
%
%

\subsection{Upper Bound  on  $\Gamma$ via $n$-digit constants $\arr_n$: Proof of Theorem \ref{th1N1}.} 
\label{sec41n}

It is known that the number of nonzero ternary digits in $(2^n)_3$ goes to infinity
as $n \to \infty$, i.e. for each $k \ge 2$ there are only finitely many $n$ with $(2^n)_3$ having
at most $k$ nonzero ternary digits. 
This result was first established in 1971 by Senge and Straus, see  \cite{SS71}.
In 1980 Colin L. Stewart \cite[Theorem 1]{St80}
obtained a quantitative refinement of such bounds. 
We obtain as a special case of his result  the  following quantitative version of the rate of growth of the number
of nonzero digits.

 
\begin{thm}\label{thA1} {\rm (C. L. Stewart) }
For each $k \ge 1$, there are only finitely many $n$ such that the
base $3$ expansion of $2^n$ (equivalently the
$3$-adic expansion $(2^n)_3$) has at most $k$
nonzero digits. More precisely, if $n_3(n)$ denotes the sum
of the base $3$ digits of $n$, then for $m \ge 25$,
$$
n_3(2^m) > \frac{\log m}{\log\log m + c} - 3,
$$
where $c>0$ is an effectively computable constant.
\end{thm}

\begin{proof}
The result  follows from \cite[Theorem 1]{St80}, taking 
for bases  $a=2$, $b=3$, and digits $\alpha=\beta=0$.
Using Stewart's notation, $L_{a, \alpha}(2^m)=2,$ so that $L_{a, \alpha, b, \beta}(2^m)-2$
counts the number of nonzero ternary digits $n_3(2^m)$ of $2^m$.
\end{proof}

We can now prove Theorem \ref{th1N1}.


\begin{proof}[ Proof of Theorem \ref{th1N1}.]
For each $n \ge 1$ we have
\[
\Gamma \le \dim_{H}(\mathcal{E}_1^{(n+1)}).
\]
We also have the inclusions 
\begin{equation}\label{eq151}
\mathcal{E}_1^{(n+1)}  = \bigcup_{0 \leq m_1 < \ldots < m_k} \mathcal{C}(1, 2^{m_1},\ldots,2^{m_{n}})
\subset \bigcup_{m=n}^{\infty} \mathcal{C}(1, 2^m),
\end{equation}
which yields
\[
\dim_{H}(\mathcal{E}_1^{(n+1)}) \le \sup_{m \ge n} \Big(\dim_{H} ( \mathcal{C}(1, 2^m))\Big).
\]
Consequently we have
\begin{equation}
\Gamma \le \sup_{m \ge n} \Big( \dim_{H}(\mathcal{C}(1, 2^m)) \Big).
\end{equation}
However Theorem \ref{thA1} implies that all $(2^m)_3$ for $m \ge n$ contain at least 
\[
k  =k(n):=  \left\lfloor \frac{\log n}{\log\log n + c}\right\rfloor - 3
\]
nonzero ternary digits. In particular
\[
 \mathcal{E}_1^{(n+1)}  
\subset \bigcup_{m=n}^{\infty} \mathcal{C}(1, 2^m) \subset \bigcup_{\{ {M}: \,n_3(M) \ge k(n)\}} \mathcal{C}(1, M).
\]
By defnition of $\arr_k$ it follows 
that
\[
\dim_{H}(\mathcal{E}_1^{(n+1)}) \le \arr_{k(n)}.
\]
Since $k(n) \to \infty$ as $n \to \infty$, we obtain
\[
\Gamma = \lim_{n \to \infty} \dim_{H}(\mathcal{E}_1^{(n+1)}) \le \lim_{k \to \infty} \arr_k,
\]
as asserted.
\end{proof}
%
%
%

\subsection{Exact bound for $\arr_2$}\label{sec42n}

We obtain a complete determination of $\arr_2$.

\begin{thm} \label{r2bound}
For all $M \ge 1$ with $M \equiv \, 1\, (\bmod \, 3)$,
one has
$$
\dim_{H} ( \mathcal{C}(1,M)) \le \log_3 \phi \approx 0.438018.
$$
where $\phi = \frac{1+ \sqrt{5}}{2}$ is the golden ratio. 
Thus $\arr_2= \log_3 \phi \approx 0.438018$
\end{thm}

\begin{proof}

We may write $M = (m_nm_{n-1} \ldots m_k 0^{k-1} 1)_3$ for some $1 \leq k \leq n < \infty$ since $M$ is an integer, $M \equiv 1 ~(\bmod 3)$. 
Our strategy will be to construct an injective map $f: \mathcal{C}(1,M) \rightarrow \mathcal{C}(1,N_k)$, where recall that $N_k = (10^{k-1}1)_3$, and by \cite[Theorem 1.8]{AL14c}, $\dim_H(\mathcal{C}(1,N_k)) = \log_3(\phi)$. Let $(\mathcal{G},v_0)$ and $(\mathcal{H}_k,w_0)$ be the right-resolving, connected, essential presentations of $\mathcal{C}(1,M)$ and $\mathcal{C}(1,N_k)$, respectively, constructed by Algorithm A of \cite{AL14c}. The injective map 
$f$ induces for each $l$ an injective map from the set of paths of length $l$ in $\mathcal{G}$ originating at $v_0$ to the set of paths of length $l$ in $\mathcal{H}_k$ originating at $w_0$, since there is a bijective correspondence between elements of $\mathcal{C}(1,M)$ or $\mathcal{C}(1,N_k)$ and infinite paths in $\mathcal{G}$ or $\mathcal{H}_k$, respectively, originating at the distinguished vertex. Thus, following \cite[Definition 1.10]{AL14a} and
 \cite[Theorem 1.1]{AL14b}, this will establish the result.

To define the map $f: \mathcal{C}(1,M) \rightarrow \mathcal{C}(1,N_k)$, we will need some notation. Let $\alpha = \ldots a_2 a_1 a_0$ be a generic element of $\mathcal{C}(1,M)$. $\alpha$ corresponds to a vertex path $\ldots v_2 v_1 v_0$ of $\mathcal{G}$ such that there is an edge labeled $a_i$ from vertex $v_i$ to vertex $v_{i+1}$. We call the digit $a_i$ \emph{restricted} if the out-degree of $v_i$ is $1$, and we call $a_i$ \emph{unrestricted} if the out-degree of $v_i$ is $2$. We call $a_i$ \emph{restricting} if $a_{i+k}$ is restricted, and otherwise we call $a_i$ \emph{non-restricting}. 

If the digit $a_i$ of $\alpha$ is unrestricted, then it is possible to find an element \newline $\alpha' = \ldots a_{i+k-1} a_{i+k-2} \ldots a_{i+1}(1-a_i)a_{i-1} \ldots a_2 a_1 a_0 \in \mathcal{C}(1,M)$. That is, changing $a_i$ to $1-a_i$ does not require us to make any other changes until the $i+k$-th digit. Then for all such $\alpha'$ the vertex $v_{i+k}'$ of the corresponding vertex path on $\mathcal{G}$ is the same. If $a_i$ is not only unrestricted but also restricting, then if this vertex $v_{i+k}'$ has out-degree $1$, we call $a_i$ \emph{unconditionally restricting}, and if $v_{i+k}'$ has out-degree $2$, we call $a_i$ \emph{conditionally restricting}. Thus, a conditionally restricting digit can be changed to become unrestricting, while an unconditionally restricting digit remains restricting when changed.

Tautologically, a conditionally restricting digit $a_i$ becomes unrestricting when replaced by $1-a_i$, but we can also see that an unrestricted, unrestricting digit $a_i$ becomes conditionally restricting when replaced by $1-a_i$, since this necessarily changes the carry digit at the $(i+k)$-th step. Thus, these types of digits come in pairs. 

Now we are ready to construct the map $f: \mathcal{C}(1,M) \rightarrow \mathcal{C}(1,N_k)$, digit-by-digit, for $\alpha \in \mathcal{C}(1,M)$:
\begin{equation}
f(\alpha)_i = 
\begin{cases}
0& \text{if } a_i \text{ is restricted or unrestricting}; \\
a_i & \text{if } a_i \text{ is unrestricted and unconditionally restricting}; \\
1 & \text{if } a_i \text{ is unrestricted and conditionally restricting}.\\
\end{cases}
\end{equation}

Though $f(\alpha)$ is clearly an element of $\Sigma_3$, we need to check first that it is really an element of $\mathcal{C}(1,N_k)$. To see this, note that if $f(\alpha)_i = 1$, then $a_i$ was restricting, so $a_{i+k}$ is restricted, thus $f(\alpha)_{i+k} = 0$. So a digit $1$ of $f(\alpha)$ is always followed, $k$ digits later, by a digit $0$. Since $\mathcal{C}(1,N_k)$ can be described as the $\mathbb{Z}/2 \mathbb{Z}$-shift of finite type with forbidden block set $\{10^{k-1}1\}$, and this block does not occur in $f(\alpha)$, we are assured that $f(\alpha) \in \mathcal{C}(1,N_k)$.

It remains only to check that $f$ is injective. Suppose $\alpha = \ldots a_2a_1a_0, \beta = \ldots b_2b_1b_0 \in \mathcal{C}(1,M)$ are distinct. Then there is a $j$ such that $a_j = 1-b_j$ and $a_i = b_i$ for all $0 \leq i < j$. Let $\ldots v_2 v_1 v_0$ and $\ldots w_2 w_1 w_0$ be the vertex paths of $\mathcal{G}$ corresponding to $\alpha$ and $\beta$, respectively. Then we must have $v_i = w_i$ for $0 \leq i \leq j$, and $v_j = w_j$ must have out-degree $2$. Thus, the digits $a_j$ of $\alpha$ and $b_j$ of $\beta$ are unrestricted. But by the discussion above, if $a_j$ is conditionally restricting then $b_j$ is unrestricting, in which case $f(\alpha)_j =1 \neq 0 = f(\beta)_j$, and vice versa, or else $a_j$ and $b_j$ are both unconditionally restricting,
 in which case $f(\alpha)_j = a_j \neq b_j = f(\beta)_j$. In any case, we see that $f(\alpha) \neq f(\beta)$, so $f$ is injective, establishing the result. 
 \end{proof}

%
%
%

\section{  Block number and  intermittency of ternary expansions}\label{sec7}

The examples given so far show that the dependence of $\dim_{H}(\sC(1, M))$ for a positive integer $M$ 
is  complicated function, 
being driven by the structure of the underlying automata, whose construction includes aspects of 
both  number theory and dynamical systems. 
One  may ask whether the Hausdorff dimension might go  to zero as a function of 
some statistic easily computable from the ternary expansion $(M)_3$. Earlier results of this paper show that  the
statistic $d_3(M)$ does not have this property.

We now  present empirical results for  two other interesting statistics of $(M)_3$:
\begin{enumerate}
\item
The {\em block number}  $\bbb_3(M)$ counts the number of blocks of consecutive nonzero
digits in the ternary expansion  $(M)_3$.
\item
The  {\em intermittency} $\sss_3(M)$  counts
 the number of distinct blocks of consecutive
matching digits in the ternary expansion $(M)_3$.
\end{enumerate}

We clearly have $\bbb_3(M) \le \sss_3(M)$. As examples, 
$$
\bbb_3( (2121011)_3) = 2; \quad \bbb_3( (2101)_3) =2,
$$
while
 $$
 \sss_3( (2121011)_3) = 6; \quad \sss_3( (2101)_3) =4.
 $$
The statistic $\bbb_3(M)$ might  be relevant to controlling the Hausdoff dimension since
 blocks of zeros at the end of the number have a simple effect on the associated automaton.
 
 Table 7.1 below presents  data on Hausdorff dimensions for  a few numbers $M$ taking the smallest values for $\sss_3(M)$,
 computed using the algorithm in Part I to six decimal places. The table also provides the number of vertices in the
 associated finite directed graph. \\

%
%

\begin{minipage}{\linewidth}
\begin{center}
\begin{tabular}{|c |c | c| c | c  | c |}
\hline
\mbox{Path Set $C(1,M)$}  & \mbox{$(M)_3$} & \mbox{$\sss_3(M)$} &\mbox{Vertices}
&
{\mbox{Perron eigenvalue}} & {Hausdorff dim}   \\
\hline
$\sC(1,10)$ & $101$ & $3$ & $4$ & $1.618033$ & $0.438018$ \\ 
$\sC(1,16)$ & $121$ & $3$ & $5$ & $1.324718$ & $0.255960$ \\ 
$\sC(1,19)$ & $201$ & $3$ & $8$ & $1.465571$ & $0.347934$ \\
$\sC(1,73)$ & $2201$ & $3$ & $16$ & $1.618033$ & $0.438018$ \\ \hline 
$\sC(1,34)$ & $1021$ & $4$ & $8$ & $1.324718$ & $0.255960$ \\
$\sC(1,46)$ & $1201$ & $4$ & $10$ & $1.112776$ & $0.097266$ \\
$\sC(1,61)$ & $2021$ & $4$ & $14$ & $1.570147$ & $0.410672$ \\
$\sC(1,64)$ & $2101$ & $4$ & $14$ & $1.357193$ & $0.278004$ \\
$\sC(1,70)$ & $2121$ & $4$ & $14$ & $1.360632$ & $0.280308$ \\ \hline
$\sC(1,91)$ & $10101$ & $5$ & $9$ & $1.465571$ & $0.347934$ \\
$\sC(1,97)$ & $10121$ & $5$ & $16$ & $1.380277$ & $0.293356$ \\
$\sC(1,100)$ & $10201$ & $5$ & $17$ & $1.354948$ & $0.276497$ \\
$\sC(1,142)$ & $12021$ & $5$ & $20$ & $1.276393$ & $0.222133$ \\
$\sC(1,145)$ & $12101$ & $5$ & $21$ & $1.000000$ & $0.000000$ \\
$\sC(1,151)$ & $12121$ & $5$ & $20$ & $1.227525$ & $0.186599$ \\
$\sC(1,172)$ & $20101$ & $5$ & $22$ & $1.288329$ & $0.230606$ \\
$\sC(1,178)$ & $20121$ & $5$ & $25$ & $1.345528$ & $0.270148$ \\
$\sC(1,181)$ & $20201$ & $5$ & $22$ & $1.324718$ & $0.255960$ \\
$\sC(1,196)$ & $21021$ & $5$ & $24$ & $1.383785$ & $0.295666$ \\
$\sC(1,208)$ & $21201$ & $5$ & $25$ & $1.290893$ & $0.232415$ \\ \hline
\end{tabular} \par
\bigskip
\hskip 0.5in {\rm TABLE 7.1.} Hausdorff  dimension of $\mathcal{C}(1,M)$ by intermittency 
\newline
\newline
\end{center}
\end{minipage}

This  extremely limited data set  exhibits  a small decrease in Hausdorff dimensions
 as the statistic $\sss_3(M)$ increases. 
It leaves open the possibility that  one might have $\dim_{H}( \mathcal{C}(1, M)) \to 0$
 as $\bbb_3(M) \to \infty$,  noting that $\bbb_3(M) \le \sss_3(M)$.
Further numerical experimentation seems warranted to get a better idea whether such an  assertion might be true.

Regarding potential applicability  of information on these statistics to the Exceptional set conjecture,  
we must point out that  it is not currently known
 whether $\bbb_3(2^n) \to \infty$ holds as $n \to \infty$  or whether $\sss_3(2^n) \to \infty$ holds 
 as $n \to \infty.$   

%
%
%
\section{Appendix A: Review of results for families $L_k= (1^{k})_3$ and $N_k= (10^{k-1}1)_3$. } \label{secA0}

We review two results proved in \cite[Section 4]{AL14c}.
The first is for the family   $L_k=  \frac{1}{2}(3^{k}-1)  = (1^{k})_3$, for $k \ge 1$,
given as \cite[Theorem 5.2]{AL14c}.


\begin{thm}\label{th17a} {\rm (Infinite Family $L_k=\frac{1}{2}(3^k-1)$)}

 (1) Let  $L_k=  \frac{1}{2}(3^{k}-1)  = (1^{k})_3$. The path set presentation $(\sG, v)$ for 
 the path set $X(1, L_k)$ underlying $\mathcal{C}(1, L_k)$
 has exactly $k$ vertices and is strongly connected. 
 
 (2) For every $k \ge 1$, 
\[
\dim_H(\mathcal{C}(1,L_k)) =  \dim_{H} \mathcal{C}(1, (1^k)_3)
= \log_3 \beta_k, 
\]
where $\beta_k $ is the unique real root greater than $1$ of $\lambda^k - \lambda^{k-1}- 1=0$.

(3) For all $k \ge 3$ there holds
\[
\dim_{H} \Big(\mathcal{C}(1,L_k)\Big) = \frac{ \log_3 k}{k} + O\left(\frac{\log\log (k)}{k}\right).
\] 
\end{thm}  

The Hausdorff dimension $\dim_{H}(\mathcal{C}(1,L_k))$ is positive but approaches
$0$ as $k \to \infty$.  We present data in Table 8.1 below.

%
%
\begin{minipage}{\linewidth}
\begin{center}
\begin{tabular}{|c |r | r | c |r |}
\hline
 \mbox{Path set} & \mbox{$L_k$} & \mbox{Vertices} & \mbox{Perron eigenvalue} &  \mbox{Hausdorff dim}  \\
\hline
 $\mathcal{C}(1,L_1)$ &  1 & 1 & $2.000000$ & $ 0.630929$ \\
$\mathcal{C}(1,L_2)$ &   4 & 2 & $1.618033$ & $ 0.438018$ \\
$\mathcal{C}(1,L_3)$ &   13& 3 & $1.465571$  &   $0.347934$  \\ 
$\mathcal{C}(1,L_4)$ &  40 & 4 & $1.380278$  & $0.293358$  \\ 
$\mathcal{C}(1,L_5)$ &  121 & 5 & $1.324718$  & $0.255960$   \\ 
$\mathcal{C}(1,L_6)$ &  364 &  6 & $1.285199$    & $0.228392$ \\
$\mathcal{C}(1,L_7)$ & 1093 &  7 & $1.255423$ & $0.207052$ \\
$\mathcal{C}(1,L_8)$ &  3280 & 8 & $1.232055$ & $0.189948$ \\ 
$\mathcal{C}(1,L_9)$ &  9841 & 9 & $1.213150$ & $0.175877$ \\ \hline
\end{tabular} \par
\bigskip
\hskip 0.5in {\rm TABLE 8.1.}  Hausdorff dimensions of $\mathcal{C}(1,L_k)$ (to six decimal places)
\newline
\newline
\end{center}
\end{minipage}

We also recall results on  the family  $N_k= 3^{k} + 1 = (10^{k-1}1)_3$,
which consists of numbers with exactly two nonzero ternary digits, with $s_3(N_k)=2$,
given as \cite[Theorem 5.5]{AL14c}.


\begin{thm}\label{th14} {\rm (Infinite Family $N_k=3^k+1$)}

(1) Let $N_k=3^k+1= (10^{k-1}1)_3$. The path set presentation $(\sG, v)$ for 
the path set $X(1, N_k)$ underlying $\mathcal{C}(1, N_k)$
 has exactly $2^k$ vertices and is strongly connected.

(2) For every integer $k \geq 1$, there holds
\[
\dim_H(\mathcal{C}(1,N_k)) =  \dim_{H} \mathcal{C}(1, (10^{k-1}1)_3)
= \log_3\bigg(\frac{1 + \sqrt{5}}{2}\bigg) \approx 0.438018.
\]
\end{thm}  

Here the Hausdorff dimension is constant as $k \to \infty$.

%
%
%

\section{Appendix B: Relation of families $P_k = (20^{k-1}1)_3$ and $L_{k+1}= (1^{k+1})_3$}\label{sec9}

We observe a relation between the Hausdorff dimensions of $\sC(1, P_k)$ and $\sC(1, L_{k+1})$.
For $1 \leq k \leq 4$, the Hausdorff dimension of $\mathcal{C}(1,(20^{k-1}1)_3)$ equals that of 
$\mathcal{C}(1,(1^{k+1})_3)$.  
 For general $k$  we obtain an inequality.
\begin{thm} \label{thm37}
The Hausdorff dimensions of  $\mathcal{C}(1,P_k)$ and $\mathcal{C}(1,L_{k+1})$ are related by
\begin{equation}\label{hdbound}
\dim_H(\mathcal{C}(1,P_k)) \geq \dim_H(\mathcal{C}(1,L_{k+1})).
\end{equation}
\end{thm}

\begin{proof} 
The marked vertex $v_0$ with label $(0)_3$   of the
path set  presentation $\mathcal{G}_{(20^{k-1}1)_3}$ associated to $\mathcal{C}(1,(20^{k-1}1)_3)$
has two exit edges, one a self-loop with edge labeled $0$,
the second an exit edge labeled $1$ to the vertex labeled $(20^{k-1})_3$. From this vertex, there is an edge labeled $1$
 to the vertex labeled $(220^{k-2})_3$. This continues for $k-2$ more steps into a vertex labeled $(2^{k})_3$, from
  which there is an out-edge labeled $1$ to a vertex labeled $(10^{k})_3$. There is a self-loop labeled $1$ at the $(10^k)_3$-vertex, 
 and a path of length $k+1$ 
  through vertices $(10^{k-j})_3$, for $1 \le j \le k$, all with edge label $0$, then 
   back to the $0$-vertex.  
   Considering only the edges given above, this  comprises a subgraph $H$
 of $\mathcal{G}_{(20^{k-1}1)_3}$ having $2k+2$ edges
 that is strongly connected, and  consists of a closed path starting and ending at $0$
 of length $2k+2$ plus two self-loops, at vertices $m=0$ and $m=3^k$. 
 (The case $k=2$ is pictured in Example \ref{example33}, where the subgraph of $\mathcal{G}_{(201)_3}$ under consideration
 is the six outer vertices in the graph in Figure \ref{fig41}.)
 Upon inspection we see that the graph $H$ is a
  double-covering of the graph  $\mathcal{G}_{(1^{k+1})_3}$ associated to $\mathcal{C}(1,L_{k+1})$ given by Algorithm A of \cite{AL14c}. This 
  implies the bound \eqref{hdbound}.  
\end{proof}

\begin{rem}
For $1 \le k \le 4$,
equality holds in Proposition \ref{thm37}
 because the subgraph of $\mathcal{G}_{(20^{k-1}1)_3}$
constructed in the proof is the strongly connected component with greatest topological entropy in these cases.
This is not true for almost all larger $k$. Theorem \ref{th17a} says $\dim_{H}(\mathcal{C}(1, L_{k})) \to 0$ as $n \to \infty$.
On the other hand  Theorem \ref{th109c} says that $\dim_{H}(\mathcal{C}(1, L_{k}))$ is bounded away from $0$
as $k \to \infty$.
\end{rem}

%
%
%


\begin{thebibliography}{9}

\bibitem{AL14a}
W. Abram and J. C. Lagarias,
\emph{Path sets in one-sided symbolic dynamics,}
Advances in Applied Mathematics, {\bf56} (2014), pp. 109-134.
\bibitem{AL14b}
W. Abram and J. C. Lagarias,
\emph{$p$-Adic path set fractals and arithmetic,}
Journal of Fractal Geometry, {\bf 1} (2014), no.1, 45-81.
\bibitem{AL14c}
W. Abram and J. C. Lagarias, 
\emph{Intersections of multiplicative translates of $3$-adic Cantor sets,}
Journal of Fractal Geometry, {\bf 1} (2014), no.4, 349--390.

\bibitem{AM79}
R.L. Adler and B. Marcus,
\textit{ Topological entropy and equivalence of dynamical systems}, 
Memoirs of the American Mathematical Society, Volume 20, No. 219, AMS: Providence, RI 1979.

\bibitem{AS03}
J.P. Alloche and J. O. Shallit,
\textit{Automatic Sequences: Theory, Applications, Generalizations},
Cambridge University Press: Cambridge 2003.

\bibitem{BH91}
M. Boyle and D. Handelman,
\emph{The spectrum of nonnegative matrices via symbolic dynamics},
Ann. Math. {\bf 133} (1991), no. 2, 249--316.

\bibitem{Edg08}
G. Edgar,
\emph{Measure, topology and fractal geometry, Second Edition}
Springer-Verlag: New York 2008.




\bibitem{Erd79}
P. Erd\H os, 
\emph{Some unconventional problems in number theory}, 
Math. Mag. {\bf 52} (1979), 67-70.

\bibitem{FT13}
E. de Faria and C. Tresser,
\emph{On Sloane's persistence problem,}
{\tt arXiv:1307.1188}, July 2013.

\bibitem{FT13b}
E. de Faria and C. Tresser,
\emph{Equidistribution of digits in powers and Diophantine approximations,}
{\tt arXiv:1307.1505}, 5 July 2013.

\bibitem{KH95}
A. Katok and B. Hasselblatt, 
\textit{Introduction to the Modern Theory of Dynamical Systems}
 (Cambridge University Press, New York, 1995).

\bibitem{Lag09}
J.C. Lagarias,
 \emph{Ternary expansions of powers of 2}, 
 J. London Math. Soc.(2) {\bf 79} (2009),  562-588.

\bibitem{Lin84}
D. Lind,
\emph{The entropies of topological Markov shifts and a related class of algebraic integers},
Ergod. Th. Dyn. Sys. {\bf 4} (1984), no. 2, 283--300.
\bibitem{LM95}
D. Lind and B. Marcus,
 \textit{An Introduction to Symbolic Dynamics and Coding,} 
 (Cambridge University Press, New York, 1995).

\bibitem{Mah61}
K. Mahler,
\textit{Lectures on diophantine approximations, Part I. $g$-adic numbers and Roth's theorem},
Prepared from notes of R. P. Bambah, University of Notre Dame Press, Notre Dame IN 1961.

\bibitem{MU03}
R. D. Mauldin and M. Urba\'{n}ski,
\emph{Graph directed Markov systems. Geometry and dynamics of limit sets,} 
Cambridge Tracts in Mathematics No. 148, Cambridge Univ. Press: Cambridge 2003.

\bibitem{MW86}
R. D. Mauldin and S. C. Williams,
\emph{On the Hausdorff dimension of some graphs,}
Trans. Amer. Math. Soc. {\bf 298} (1986), no. 2, 793--803.

\bibitem{MW88}
R.D. Mauldin and S.C. Williams, 
\emph{Hausdorff Dimension of Graph Directed Constructions}, 
Transactions of the American Mathematical Society, {\bf 309}, No. 2 (1988) , 811-829.

\bibitem{SS71}
H. G. Senge and E. Straus,
\emph{P.V. numbers and sets of multiplicity,}
Periodica Math. Hung. {\bf 3} (1973), 93--100.

\bibitem{Sim06}
J.G. Simonsen, 
\emph{On the Computabillity of the Topological Entropy of Subshifts},
 Discrete Mathematics and Theoretical Computer Science, {\bf 8} (2006), 83-96.

\bibitem{St80}
C.~L.~Stewart, 
\emph{On the Representation of an Integer in two
Different Bases,} 
J. Reine Angew. Math., {\bf 319} (1980), 63--72. 

\bibitem{Wei73}
B. Weiss,
\emph{Subshifts of finite type and sofic systems,}
Monatshefte f\"{u}r Math. {\bf 77} (1973), 462--474.

\bibitem{Wil88}
S. Williams,
\emph{A sofic system which is not spectrally of finite type,}
Ergod. Th. Dyn. Sys. {\bf 8} (1988), 483--490.
\end{thebibliography}
\end{document}